\documentclass[10pt]{amsart}

\usepackage{times,amsmath,amsbsy,amssymb,amscd,mathrsfs}
\usepackage{graphicx,subcaption,epstopdf,wrapfig,chemarrow}

\usepackage{algorithm2e} 
\usepackage{multicol,multirow}
\usepackage{mathtools}
\usepackage[usenames,dvipsnames,svgnames,table]{xcolor}
\usepackage[all]{xy}
\usepackage{wrapfig}
\usepackage{tcolorbox}

\usepackage{tikz,tikz-cd}
\usepackage[utf8]{inputenc}
\usepackage{pgfplots} 
\usepackage{pgfgantt}
\usepackage{pdflscape}
\pgfplotsset{compat=newest} 
\pgfplotsset{plot coordinates/math parser=false}
\newlength\fwidth

\definecolor{myBlue}{rgb}{0.0,0.0,0.55}
\definecolor{green}{rgb}{0.0,0.7,0.2}
\usepackage[pdftex,colorlinks=true,citecolor=myBlue,linkcolor=myBlue]{hyperref}

\usepackage[hyperpageref]{backref}

\usepackage{comment,enumerate,multicol,xspace}

  \newcounter{mnote}
  \let\oldmarginpar\marginpar
    \renewcommand\marginpar[1]{\-\oldmarginpar[\raggedleft\footnotesize #1]%
    {\raggedright\footnotesize #1}}

\newtheorem{theorem}{Theorem}[section]
\newtheorem{lemma}[theorem]{Lemma}

\newtheorem{proposition}[theorem]{Proposition}

\newtheorem{remark}[theorem]{Remark}

\newcommand{\dd}{\,{\rm d}}

\newcommand{\bs}{\boldsymbol}

\newcommand{\curl}{{\rm curl\,}}
\renewcommand{\div}{\operatorname{div}}

\newcommand{\sym}{\operatorname{sym}}

\newcommand{\vertiii}[1]{{\left\vert\kern-0.25ex\left\vert\kern-0.25ex\left\vert #1 
    \right\vert\kern-0.25ex\right\vert\kern-0.25ex\right\vert}}

\newcommand{\Oplus}{\ensuremath{\vcenter{\hbox{\scalebox{1.5}{$\oplus$}}}}}

\newcommand{\graysquare}{\tikz \fill[gray!70] (0,0) rectangle (0.18,0.18);}

\usepackage{booktabs}

\begin{document}
\title[Planar Elasticity Complexes and $C^1$ Elements]{Explicit Planar Finite Element Elasticity Complexes and $C^1$ Elements on Barycentric Refinements}

\author{Chunyu Chen}
\address{School of Mathematical Sciences, Peking University, Beijing 100871, China}
\email{cbtxs@math.pku.edu.cn}

\author {Long Chen}
\address{Department of Mathematics, University of California at Irvine, Irvine, CA 92697 USA }
\email{chenlong@math.uci.edu}

\author {Xuehai Huang}
\address{School of Mathematics, Shanghai University of Finance and Economics, Shanghai 200433, China}
\email{huang.xuehai@sufe.edu.cn}

\thanks{The first author was supported by the China Postdoctoral Science Foundation (Grant No. 2025M783064). The second author was supported by NSF DMS-2309785. The third author was supported by the National Natural Science Foundation of China (Grant No. 12671432).}

\makeatletter
\@namedef{subjclassname@2020}{\textup{2020} Mathematics Subject Classification}
\makeatother

\maketitle

\begin{abstract}
We give a closed-form realization of the finite element elasticity complex associated with the Arnold--Douglas--Gupta family on barycentric refinements. Their construction implies the existence of minimal $C^1$ potential spaces, but does not provide explicit general-degree scalar enrichments. On each macrotriangle, we derive three divergence-free stress enrichments and their Airy potentials. For every degree $q\geq4$, the resulting potential space enriches $\mathbb P_q$ by three explicit functions. A geometric decomposition yields unisolvent degrees of freedom and $C^1$ assembly without vertex supersmoothness. The space attains the dimension bound for macroelements containing $\mathbb P_q$ with full polynomial boundary trace freedom. Its Airy image is the divergence-free subspace of the corresponding stress space, and global assembly gives an exact discrete elasticity complex on simply connected domains. The low-order spaces $U_2$ and $U_3$ recover the reduced and full Hsieh--Clough--Tocher spaces, respectively.
\end{abstract}

\section{Introduction}
Mixed finite element methods for linear elasticity approximate the stress $\boldsymbol{\sigma}\in H(\div,\Omega;\mathbb S)$ and the displacement $\boldsymbol{u}\in L^2(\Omega;\mathbb{R}^2)$ in the Hellinger--Reissner formulation, where the symmetry constraint $\boldsymbol{\sigma}=\boldsymbol{\sigma}^{\intercal}$ is imposed directly. The macroelement construction of Arnold--Douglas--Gupta (ADG)~\cite{ArnoldDouglasGupta1984} gives a family of symmetric $H(\div)$-conforming stress spaces $\Sigma_{k,h}^{\rm ADG}$ on the barycentric refinement of triangles. On each macrotriangle $T$, split into three subtriangles, the stress space enriches $\mathbb{P}_k(T;\mathbb{S})$ by three functions, a construction denoted by $\mathbb{P}_k(T;\mathbb{S})+3$. The three divergence-free enrichments allow independent degree-$k$ normal tractions $\bs\sigma\bs n$ on the three edges while preserving polynomial divergence. 

The Arnold--Douglas--Gupta element fits into the two-dimensional elasticity complex
\begin{equation}\label{eq:2Dcomplex}
\mathbb P_1 \hookrightarrow H^2(\Omega) \xrightarrow{J} H(\div,\Omega;\mathbb S) \xrightarrow{\div} L^2(\Omega;\mathbb R^2)\to 0,
\end{equation}
where $J$ is the Airy operator (rotated Hessian), defined in \eqref{eq:Jdef}. Discrete versions of this complex can be constructed by finite element exterior calculus (FEEC) \cite{Arnold2018,ArnoldFalkWinther2006} or by the Bernstein--Gelfand--Gelfand (BGG) framework~\cite{Arnold;Hu:2020Complexes,ChenHuang2022Complexes,Chen;Huang:2022femcomplex2d}. Recent work includes low-order elasticity complexes on barycentric refinements~\cite{ChristiansenHu2023}, discrete elasticity complexes on barycentric splits in three dimensions~\cite{GuzmanHuang2026,ChristiansenGopalakrishnanGuzmanHu2024}, and finite element elasticity complexes on non-refined triangulations~\cite{Chen;Huang:2021Finite,ChenHuang2022Complexes,Chen;Huang:2022femcomplex2d}.

ADG~\cite[Section 4]{ArnoldDouglasGupta1984} proved an exact local discrete elasticity sequence for the full composite spaces. Arnold and Winther~\cite{ArnoldWinthe2003finite} subsequently discussed the connection with composite $H^2$ elements, referring to the Douglas--Dupont--Percell--Scott (DDPS) spaces~\cite{DouglaDupontPercelScott1979family}.

The ADG construction uses linear algebra and symmetry, with stress coefficients tabulated for the lowest-order case.
Our contribution is an explicit realization of the ADG stress spaces and their scalar potential spaces. We derive closed-form barycentric formulas for three divergence-free enrichment stresses $\bs\psi_0,\bs\psi_1,\bs\psi_2$ and their Airy potentials for every $k\geq2$. On each macrotriangle $T$, the selected local stress space is
\begin{equation}\label{intro:ADGstress}
\Sigma_{k}^{\rm ADG}(T):=\mathbb P_k(T;\mathbb S)\oplus {\rm span}\{\bs\psi_0,\bs\psi_1,\bs\psi_2\},\quad \div\bs\psi_i=0,\ i=0,1,2,\ k\geq2.
\end{equation}
The existence of $\mathbb P_{k+2}+3$ potential spaces can be deduced from the ADG construction and the local elasticity complex.
We construct explicit Airy potentials for these enrichments: for each $i=0,1,2$, there is a function $v_i\in C^1(T)$ such that
$J(v_i)=\bs\psi_i.$
These potentials define the local $C^1$ space
\begin{equation}\label{intro:Uq}
U_{k+2}(T):=\mathbb P_{k+2}(T)\oplus {\rm span}\{v_0,v_1,v_2\},\qquad k\geq2.
\end{equation}

By assembling the vertex and edge degrees of freedom \eqref{HCTfeDoF1}--\eqref{HCTfeDoF4}, we obtain a global space $U_{k+2,h}\subset C^1(\Omega)$ and the following explicit discrete elasticity complex:
\begin{equation}\label{intro:femelasticitycomplexHCT}
\mathbb P_1\hookrightarrow U_{k+2,h} \xrightarrow{J} \Sigma_{k,h}^{\rm ADG}
\xrightarrow{\div} V_{k-1,h}\to 0,\qquad k\geq2,
\end{equation}
where $\Sigma_{k,h}^{\rm ADG}$ is obtained by assembling the local stress spaces in \eqref{intro:ADGstress}, and $V_{k-1,h}$ is the discontinuous $\mathbb R^2$-valued polynomial space of degree $k-1$. The lowest-order case, $U_4$--$\Sigma_2^{\rm ADG}$--$P_1$, is shown in Fig.~\ref{fig:lowest}.

\begin{figure}[htbp]
\begin{center}
\includegraphics[width=4in]{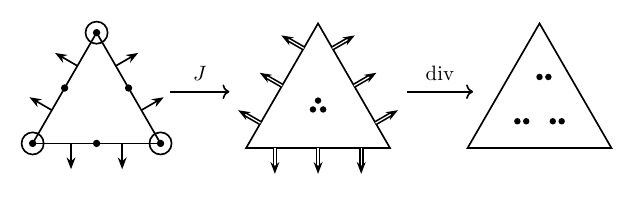}
\caption{The finite element elasticity complex in the lowest-order case $U_4$--$\Sigma_2$--$P_1$.}
\label{fig:lowest}
\end{center}
\end{figure}

The stress family considered here has $k\geq2$, with potential space $U_q$ for $q=k+2\geq4$. The cubic Hsieh--Clough--Tocher (HCT) space $U_3(T)$ is also the Airy potential space of the classical Johnson--Mercier complex~\cite{JohnsonMercier1978,ChristiansenHu2023}. The classical elasticity complexes associated with HCT and reduced HCT are recalled in Remark~\ref{rem:low-order-complexes}.

For $k\ge3$, the Hu--Zhang element~\cite{HuZhang2014,Hu2015a} admits the discrete elasticity complex~\cite{christiansenNodalFiniteElement2018}:
\begin{equation*}
\mathbb P_1\hookrightarrow U_{k+2,h}^{\rm Arg} \xrightarrow{J} \Sigma_{k,h}^{\rm HZ}
\xrightarrow{\div} V_{k-1,h}\to 0,
\end{equation*}
where $U_{k+2,h}^{\rm Arg}$ is an Argyris-type $C^1$ space requiring $C^2$ continuity at vertices. The Arnold--Winther family~\cite{ArnoldWinther2002} also uses Argyris-type potentials.
Consequently, both stress spaces are continuous at vertices. Our $C^1$ potential space requires no $C^2$ continuity at vertices. The stress space is assembled through edge traction DoFs, without shared vertex DoFs, allowing a hybridizable formulation.

\begin{figure}[htbp]
\begin{center}
\includegraphics[width=3.8in]{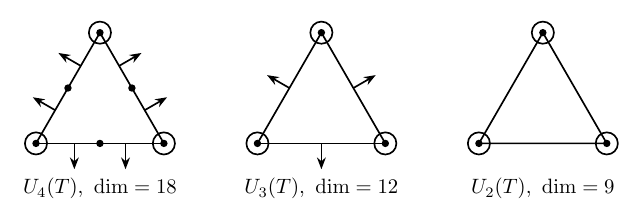}
\caption{Degrees of freedom for the $C^1$ elements $U_4(T)$, $U_3(T)$, and $U_2(T)$, with dimensions $18$, $12$, and $9$, respectively.}
\label{fig:C1low}
\end{center}
\end{figure}

The Airy potentials also yield an explicit family of composite $C^1$ finite elements on barycentric refinements. The low-order spaces recover classical Clough--Tocher constructions: $U_3(T)$ is the cubic Hsieh--Clough--Tocher element~\cite{Hsieh1962,CloughTocher1965}, while the nine-dimensional space $U_2(T)$ is the reduced cubic space, determined only by the function value and gradient at the three vertices. Figure~\ref{fig:C1low} shows their DoFs alongside those of $U_4(T)$.

For degree $q=k+2\geq4$, the space $U_q(T)$ in \eqref{intro:Uq} gives an explicit three-function enrichment of $\mathbb P_q(T)$ on the barycentric refinement. The assembled space has dimension
$$
\dim U_{q,h}
=
3|\Delta_0(\mathcal T_h)|
+
(2q-5)|\Delta_1(\mathcal T_h)|
+
\binom{q-4}{2}|\Delta_2(\mathcal T_h)|.
$$
In Theorem~\ref{th:minimal-C1-macroelement}, we show that this dimension is minimal among $C^1$ macroelement spaces that contain $\mathbb P_q$, impose no vertex supersmoothness, and allow the function value $u|_e\in\mathbb P_q(e)$ and the normal derivative $\partial_nu|_e\in\mathbb P_{q-1}(e)$ to be prescribed simultaneously on all three edges, subject only to agreement with common vertex values and gradients. 

Higher-order $C^1$ Clough--Tocher macroelements without vertex supersmoothness were constructed by Douglas--Dupont--Percell--Scott~\cite{DouglaDupontPercelScott1979family}. Their construction uses the full piecewise polynomial space, whereas $U_q(T)$ enriches $\mathbb P_q(T)$ by three functions. For $q=4,5,6$, their local dimensions are $21,33,48$, compared with $18,24,31$ for $U_q(T)$.

The family also has fewer DoFs than other Clough--Tocher spaces. Gro{\v{s}}elj and Knez~\cite{grovselj2022generalized} constructed $C^1$ spaces with supersmoothness at the interior split point and nonnegative, locally supported partition-of-unity bases in Bernstein--B\'ezier form. Their degree-$q$ space has $(2q-7)|\Delta_2(\mathcal T_h)|$ more DoFs than $U_{q,h}$. Brubeck and Kirby~\cite{brubeck2026fiat} implemented classical and higher-order HCT-type macroelements in FIAT and Firedrake; in particular, their quartic HCT-type element has local dimension $19$, while $U_4(T)$ has dimension $18$. 

The boundary DoFs coincide with those of related $C^1$ virtual elements~\cite{ChenHuangWei2022,BrezziMarini2013,BeiraodaVeigaManzini2014}, while our elements use fewer interior moments; see Section~\ref{sec:vem}. Here the basis functions are explicit on the barycentric refinement, so no stabilization is needed. The quartic space $U_4(T)$ is also a conforming analogue of the $H^2$-nonconforming quartic element in~\cite{nilssen2001robust}. 

Section~\ref{sec:symstress} constructs the symmetric stress spaces. Section~\ref{sec:C1fe} develops their scalar potentials and the associated $C^1$ elements. Section~\ref{sec:feelascomplex} establishes the discrete elasticity complexes and their exactness on simply connected domains.

\section{Symmetric finite element stress spaces on barycentric refinements}\label{sec:symstress}
We construct symmetric stress spaces on barycentric refinements using explicit barycentric formulas for the local enrichments.

\subsection{Notation}
Let $\Omega\subset\mathbb R^2$ be a bounded connected polygonal Lipschitz domain, and let $\{\mathcal T_h\}$ be a shape-regular family of conforming triangulations of $\Omega$. Polynomial degrees are fixed. Constants in $\lesssim$ may depend on $\Omega$, the degree, and the shape-regularity bound, but not on $h$. We set $\mathbb P_m=\{0\}$ for $m<0$ and use $\binom{m}{n}=0$ for integers $0\leq m<n$. For each $T \in \mathcal{T}_h$, let $T^{\rm R}$ be its barycentric refinement, obtained by connecting the barycenter $\texttt{v}_c$ of $T$ to its three vertices $\texttt{v}_0$, $\texttt{v}_1$, and $\texttt{v}_2$, where $\texttt{v}_c=(\texttt{v}_0+\texttt{v}_1+\texttt{v}_2)/3$. This splits $T$ into three subtriangles $T_0$, $T_1$, and $T_2$, where $T_i$ is opposite to $\texttt{v}_i$ for $i=0,1,2$. For a subtriangle $T_i \subset T^{\rm R}$, let $\chi_{T_i}$ be its characteristic function. We write $\boldsymbol{t}_{i,j}=\texttt{v}_j-\texttt{v}_i$ for $i,j=0,1,2,c$. See Fig.~\ref{fig:2Dsplit}.

The vertices $\texttt v_0,\texttt v_1,\texttt v_2$ are numbered counterclockwise. Let $\lambda_i$ be the barycentric coordinate on $T$ associated with the vertex $\texttt{v}_i$, and let $\lambda_i^{\rm R}$ be the continuous piecewise linear nodal basis function on $T^{\rm R}$ that equals one at $\texttt v_i$ and zero at the other vertices, including $\texttt v_c$. On each subtriangle containing $\texttt v_i$, its restriction is the barycentric coordinate associated with $\texttt v_i$; it vanishes on $T_i$. Checking the vertex values gives
\begin{equation}\label{eq:lambdaRprop1}
\lambda_i^{\rm R}=\lambda_i-\lambda_j\quad\text{on }T_j,\qquad i,j=0,1,2.
\end{equation}
Consequently,
\begin{equation}\label{eq:lambdaRprop2}
(\lambda_2-\lambda_1)|_{T_1\cup T_2} = \lambda_2^{\rm R}\chi_{T_1}-\lambda_1^{\rm R}\chi_{T_2}, \quad \nabla\lambda_0^{\rm R}|_{T_1}=\nabla\lambda_0-\nabla\lambda_1.
\end{equation}

\begin{figure}[htbp]
\centering
\begin{subfigure}[b]{0.5\linewidth}
\centering
\includegraphics[width=4cm]{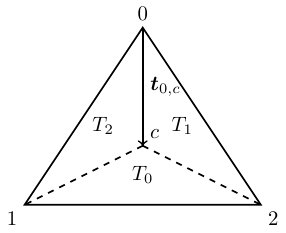}
\caption{Barycentric refinement.}
\end{subfigure}\hfill
\begin{subfigure}[b]{0.45\linewidth}
\centering
\includegraphics[width=0.85\linewidth]{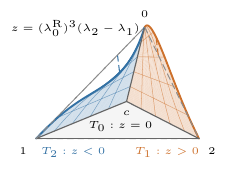}
\caption{Graph of $(\lambda_0^{\rm R})^3(\lambda_2-\lambda_1)$.}
\end{subfigure}
\caption{Barycentric refinement of a triangle and a quartic Airy potential up to a constant.}
\label{fig:2Dsplit}
\end{figure}

Let $\Delta_\ell(T)$ denote the set of $\ell$-dimensional subsimplices of $T$, for $0\leq\ell\leq2$. We also identify a subsimplex $e$ with a subset of $\{0,1,2\}$ and write $e^*$ for its complement. We label the edges of $T$ by $e_i$, where $e_i$ is the edge opposite to the vertex $\texttt{v}_i$, that is, $e_i=\{i\}^*$ for $i=0,1,2$. 

Throughout, $|T|$ denotes the area of $T$. For an edge $e$, $\boldsymbol{t}_e$ and $\boldsymbol{n}_e$ denote the unit tangent and normal vectors, oriented according to the fixed vertex ordering. On $e_i$, we write $\boldsymbol n_i=\boldsymbol n_{e_i}$. Let $b_T=\lambda_0\lambda_1\lambda_2$ be the element bubble and let $b_e=\lambda_i\lambda_j$ be the edge bubble for $e=\{i,j\}$.

We write $\mathbb S$ for the space of $2\times2$ symmetric matrices, and the divergence operator $\div$ is applied row-wise to tensor fields. For a matrix $A$, we write $\sym A=(A+A^{\intercal})/2$.

\subsection{Stress element}
For $k\ge 2$, we define the local symmetric stress space by
\begin{equation*}
\Sigma_{k, \psi}^{\operatorname{div}}(T; \mathbb{S})
=
\mathbb{P}_k(T; \mathbb{S})
+
\operatorname{span}\{\boldsymbol{\psi}^k_{0},\boldsymbol{\psi}^k_{1},\boldsymbol{\psi}^k_{2}\},
\end{equation*}
where
\begin{equation}\label{eq:2d}
\begin{aligned}
\boldsymbol{\psi}^k_{0}
&=
\Bigl[
2(\lambda_0^{\rm R})^{k}\,\sym(\boldsymbol{t}_{c,0}\otimes \boldsymbol{t}_{c,1})
- k(\lambda_0^{\rm R})^{k-1}\lambda_1^{\rm R}\,\boldsymbol{t}_{c,1}\otimes \boldsymbol{t}_{c,1}
\Bigr]\chi_{T_2}
\\&\quad
+
\Bigl[
-2(\lambda_0^{\rm R})^{k}\,\sym(\boldsymbol{t}_{c,0}\otimes \boldsymbol{t}_{c,2})
+ k(\lambda_0^{\rm R})^{k-1}\lambda_2^{\rm R}\,\boldsymbol{t}_{c,2}\otimes \boldsymbol{t}_{c,2}
\Bigr]\chi_{T_1},
\end{aligned}
\end{equation}
and $\boldsymbol{\psi}^k_{1}$ and $\boldsymbol{\psi}^k_{2}$ are obtained from $\boldsymbol{\psi}^k_0$ by cyclic permutation of the indices $(0,1,2)$. Each enrichment $\boldsymbol\psi_i^k$ vanishes on $T_i$. To check $H(\div)$ conformity, consider $\boldsymbol\psi_0^k$. Since $k\geq2$ and $\lambda_0^{\rm R}=0$ on $[\texttt v_1,\texttt v_c]$ and $[\texttt v_2,\texttt v_c]$, both traces of $\boldsymbol\psi_0^k$ vanish on these edges. On the remaining interior edge $[\texttt v_0,\texttt v_c]$, we have $\lambda_1^{\rm R}=\lambda_2^{\rm R}=0$, so
\begin{equation}\label{eq:difference}
\boldsymbol\psi_0^k|_{T_2}-\boldsymbol\psi_0^k|_{T_1}
=2(\lambda_0^{\rm R})^k \sym(\boldsymbol{t}_{c,0}\otimes (\boldsymbol{t}_{c,1} + \boldsymbol{t}_{c,2}))=-2(\lambda_0^{\rm R})^k\boldsymbol t_{c,0}\otimes\boldsymbol t_{c,0},
\end{equation}
using $\boldsymbol t_{c,1}+\boldsymbol t_{c,2}=-\boldsymbol t_{c,0}$. This jump has zero normal trace because $\boldsymbol t_{c,0}$ is tangent to the edge. Thus $\boldsymbol\psi_0^k\in H(\div,T;\mathbb S)$, and cyclic permutation gives the same result for $\boldsymbol\psi_1^k,\boldsymbol\psi_2^k$. Unlike the enrichments in~\cite{ChenHuang2025}, these functions are divergence-free, as verified by the explicit Airy potentials in Section~\ref{sec:localenrichment}. 

On an edge $e\in \Delta_1(T)$, we use the $t$-$n$ decomposition of $\mathbb S$:
\begin{equation*}
\begin{aligned}
\mathscr T^e(\mathbb S) &= \operatorname{span}\{\boldsymbol{t}_e\otimes \boldsymbol{t}_e\}, \\
\mathscr N^e(\mathbb S) &= \operatorname{span}\{\boldsymbol{n}_e\otimes \boldsymbol{n}_e,\ \sym(\boldsymbol{n}_e\otimes \boldsymbol{t}_e)\}.
\end{aligned}
\end{equation*}
These are the tangential and normal subspaces associated with $e$.

Let $\operatorname{tr}^{\operatorname{div}}\bs\tau=(\bs\tau\bs n)|_{\partial T}$ denote the normal trace, where $\bs n$ is the outward unit normal. We next introduce the $\div$-bubble polynomial space
\begin{equation*}
\mathbb B_k^{\operatorname{div}}(T; \mathbb{S})
=
\mathbb P_k(T; \mathbb{S})\cap \ker(\operatorname{tr}^{\operatorname{div}})
=
\mathbb P_k(T; \mathbb{S})\cap H_0(\div, T;\mathbb S),
\end{equation*}
which admits the geometric characterization \cite{ChenHuang2024,ChenHuang2025}
\begin{equation*}
\mathbb B_k^{\operatorname{div}}(T; \mathbb{S})
=
b_T \mathbb{P}_{k-3}(T; \mathbb{S})
\;\oplus\;
\Oplus_{e\in\Delta_1(T)} b_e \mathbb{P}_{k-2}(e; \mathscr T^e(\mathbb S)).
\end{equation*}
A polynomial in $\mathbb P_m(e)$, $e=\{i,j\}$, is written uniquely as a homogeneous polynomial of degree $m$ in the endpoint barycentric coordinates $\lambda_i,\lambda_j$ and extended to $T$ using the same expression. This convention is applied componentwise to tensor-valued polynomials.
Another useful characterization, due to Hu and Zhang~\cite{HuZhang2014,Hu2015a}, is
\begin{equation}\label{eq:HZbubble}
\mathbb B_k^{\operatorname{div}}(T; \mathbb{S})
=
\mathbb P_{k-2}(T)\otimes
\operatorname{span} \{ b_e \boldsymbol t_e\otimes \boldsymbol t_e : e\in \Delta_1(T)\}.
\end{equation}
Therefore, when defining degrees of freedom, it is enough to work with $\mathbb P_{k-2}(T;\mathbb S)$.

For $k\geq2$, this characterization gives $\dim\mathbb B_k^{\operatorname{div}}(T;\mathbb S)=3\binom{k}{2}$. Since this bubble space is the kernel of the normal-traction map on $\mathbb P_k(T;\mathbb S)$, the image of that map has dimension
\[
3\binom{k+2}{2}-3\binom{k}{2}=6k+3.
\]
In contrast, independent degree-$k$ vector-valued traction data $\bs \sigma\bs n$ on the three edges have dimension $3\cdot2(k+1)=6k+6$. Thus at least three enrichment functions are needed to allow all such edge data.

We also use the following geometric decomposition of $\mathbb P_k(T;\mathbb S)$ \cite{ChenHuang2024,ChenHuang2025}:
\begin{equation}\label{eq:dec}
\mathbb P_k(T;\mathbb S)
=
\mathbb B_k^{\operatorname{div}}(T; \mathbb{S})
\;\oplus\;
\mathbb P_1(T;\mathbb S)
\;\oplus\;
\Oplus_{e\in\Delta_1(T)} b_e \mathbb{P}_{k-2}(e; \mathscr N^e(\mathbb S)).
\end{equation}

\begin{lemma}\label{lem:geodecompSigmak}
For $k\geq 2$, the enrichment is direct:
\begin{equation*}
\Sigma_{k, \psi}^{\operatorname{div}}(T; \mathbb{S})
=
\mathbb{P}_k(T; \mathbb{S})
\;\oplus\;
\operatorname{span}\{\boldsymbol{\psi}^k_{0},\boldsymbol{\psi}^k_{1},\boldsymbol{\psi}^k_{2}\}.
\end{equation*}
Moreover,
\begin{equation}\label{geodecomp:Sigmak}
\Sigma_{k, \psi}^{\operatorname{div}}(T; \mathbb{S})
=
\mathbb B_k^{\operatorname{div}}(T; \mathbb{S})
\;\oplus\;
\widetilde{\Sigma}_{1, \psi}^{\operatorname{div}}(T; \mathbb{S})
\;\oplus\;
\Oplus_{e\in\Delta_1(T)} b_e \mathbb{P}_{k-2}(e; \mathscr N^e(\mathbb S)),
\end{equation}
where, for $k\geq2$,
\begin{equation*}
\widetilde{\Sigma}_{1, \psi}^{\operatorname{div}}(T; \mathbb{S})
=
\mathbb{P}_1(T; \mathbb{S})
\;\oplus\;
\operatorname{span}\{\boldsymbol{\psi}^k_{0},\boldsymbol{\psi}^k_{1},\boldsymbol{\psi}^k_{2}\}.
\end{equation*}
\end{lemma}

\begin{proof}
Assume
\begin{equation}\label{eq:20260205}
\boldsymbol\tau+c_0\boldsymbol\psi_0^k+c_1\boldsymbol\psi_1^k+c_2\boldsymbol\psi_2^k=0,
\end{equation}
where $\boldsymbol\tau\in\mathbb P_k(T;\mathbb S)$. By \eqref{eq:difference},
\[
\bigl[(\boldsymbol\psi_0^k|_{T_2}-\boldsymbol\psi_0^k|_{T_1})(\texttt v_0)\bigr]\boldsymbol t_{0,c}
=-2|\boldsymbol t_{0,c}|^2\boldsymbol t_{0,c}\ne0.
\]
Thus $\boldsymbol\psi_0^k\boldsymbol t_{0,c}$ has different values at the same vertex $\texttt v_0$ from $T_1$ and $T_2$. Since $\boldsymbol\tau$ is continuous and $\boldsymbol\psi_1^k,\boldsymbol\psi_2^k$ vanish there for $k\geq2$, \eqref{eq:20260205} implies $c_0=0$. Cyclic permutation gives $c_1=c_2=0$, hence $\boldsymbol\tau=0$. The enrichment is therefore direct, and \eqref{geodecomp:Sigmak} follows from \eqref{eq:dec}.
\end{proof}

With $k\geq 2$ fixed, we write $\bs\psi_i$ for $\bs\psi_i^k$ to simplify notation.

\begin{theorem}[Unisolvence]\label{thm:unisolv_sigma_kpsi}
For $k\geq2$, the degrees of freedom
\begin{subequations}\label{DoF}
\begin{align}\label{eq:DoF1}
\int_e (\bs \sigma \bs n)\cdot \bs q \,{\rm d}s, &
\quad \bs q\in \mathbb P_{k}(e; \mathbb R^2), \ e\subset\partial T,\\
\label{eq:DoF2}
\int_T \bs \sigma: \bs \tau \,{\rm d}x, &
\quad \bs \tau \in \mathbb P_{k-2}(T; \mathbb S),
\end{align}
\end{subequations}
are unisolvent for
\[
\Sigma_{k,\psi}^{\div}(T;\mathbb S)
=
\mathbb P_k(T;\mathbb S)\oplus {\rm span}\{\boldsymbol\psi_0,\boldsymbol\psi_1,\boldsymbol\psi_2\}.
\]
\end{theorem}
\begin{proof}
Suppose all DoFs of $\bs\tau\in\Sigma_{k,\psi}^{\div}(T;\mathbb S)$ vanish, and write
\[
\bs\tau=\bs p+\sum_{i=0}^2 a_i\bs\psi_i^k,
\qquad \bs p\in\mathbb P_k(T;\mathbb S).
\]
The edge moments \eqref{eq:DoF1} imply
\[
(\bs\tau|_{T_1})\bs n_1=0\quad\text{on }e_1,
\qquad
(\bs\tau|_{T_2})\bs n_2=0\quad\text{on }e_2.
\]
These polynomial trace identities also hold at their common vertex $\texttt v_0$. By symmetry,
\[
\bs n_1^{\intercal}(\bs\tau|_{T_2}-\bs\tau|_{T_1})(\texttt v_0)\bs n_2=0.
\]
Evaluating \eqref{eq:difference} at $\texttt v_0$, where $\lambda_0^{\rm R}(\texttt v_0)=1$, and noting that $\bs p$ has no jump and $\bs\psi_1^k,\bs\psi_2^k$ vanish there, gives
\[
0=-2a_0(\bs n_1\cdot\boldsymbol t_{c,0})(\bs n_2\cdot\boldsymbol t_{c,0}).
\]
Both factors are nonzero because $[\texttt v_0,\texttt v_c]$ is parallel to neither $e_1$ nor $e_2$. Hence $a_0=0$. Cyclic permutation gives $a_1=a_2=0$.

Thus $\bs\tau=\bs p\in\mathbb B_k^{\operatorname{div}}(T;\mathbb S)$, and the vanishing interior moments \eqref{eq:DoF2} imply $\bs\tau=0$ using the bubble characterization \eqref{eq:HZbubble}. Since the number of DoFs equals the dimension of the space by \eqref{geodecomp:Sigmak}, this proves unisolvence.
\end{proof}

Let $\mathcal E_h$ be the set of edges of $\mathcal T_h$. Fix one unit normal $\bs n_e$ on each edge $e$, and use this same normal in the edge moments \eqref{eq:DoF1} for both adjacent triangles. Define the global stress space by requiring these moments to be single-valued across interelement edges. We set
\[
\begin{aligned}
\Sigma_{k,h}
:=
\Bigl\{\bs\sigma\in L^2(\Omega;\mathbb S):
&\ \bs\sigma|_T\in \Sigma_{k,\psi}^{\div}(T;\mathbb S)\ \forall\,T\in\mathcal T_h,\\
&
\ \text{and \eqref{eq:DoF1} is single-valued on each }e\in\mathcal E_h
\Bigr\}.
\end{aligned}
\]
Then $\Sigma_{k,h}\subset H(\div,\Omega;\mathbb S)$.

\begin{remark}
$\Sigma_{k,h}$ is a symmetric-tensor analogue of the Brezzi--Douglas--Marini (BDM) $H(\div)$-conforming vector elements \cite{BrezziDouglasMarini1985,BrezziDouglasDuranFortin1987}. This construction is not possible with $\mathbb P_k(T; \mathbb S)$ alone, but becomes possible after enrichment.

Increasing the number of interior degrees of freedom also leads to a
Raviart--Thomas (RT)-type element \cite{RaviartThomas1977}.
The edge moments \eqref{eq:DoF1} remain unchanged, while the interior moments \eqref{eq:DoF2} are replaced by
\[
\int_T \bs \sigma : \bs \tau \,{\rm d}x,
\quad \bs\tau\in \mathbb P_{k-1}(T;\mathbb S).
\]
For $k\geq2$, the corresponding local shape function space is
\[
\Sigma_{k+1^{-}}(T; \mathbb{S})
=
\mathbb B_{k+1}^{\operatorname{div}}(T; \mathbb{S})
\oplus \widetilde{\Sigma}_{1, \psi}^{\operatorname{div}}(T; \mathbb{S})
\oplus \Oplus_{e\in\Delta_1(T)} b_e \mathbb{P}_{k-2}(e; \mathscr N^e(\mathbb S)).
\]
Thus only the bubble space is enriched, from $\mathbb B_k^{\operatorname{div}}(T;\mathbb S)$ to $\mathbb B_{k+1}^{\operatorname{div}}(T;\mathbb S)$. We denote the assembled RT-type space by $\Sigma_{k+1^{-},h}$.
\end{remark}

\section{$C^1$ finite element spaces on barycentric refinements}\label{sec:C1fe}
We construct scalar $C^1$ finite element spaces whose Airy images are the divergence-free subspaces of the stress spaces. The Airy potentials are defined for $k\geq1$; for $k\geq2$, they are linearly independent modulo $\mathbb P_{k+2}(T)$.

\subsection{Local Airy potential}\label{sec:localenrichment}
In two dimensions, we define the rotated Hessian, also called the Airy operator, by
\begin{equation}
\label{eq:Jdef}
J(v)
=
\begin{pmatrix}
0 & -1\\
1 & 0
\end{pmatrix}
\nabla^2 v
\begin{pmatrix}
0 & 1\\
-1 & 0
\end{pmatrix}
=
\begin{pmatrix}
\partial_{yy} v & -\partial_{yx} v\\
-\partial_{xy} v & \partial_{xx} v
\end{pmatrix}.
\end{equation}
The operator satisfies $\div J(v)=0$. 

We construct an $H^2$-conforming scalar element on $T^{\rm R}$. 
The factor $(\lambda_0^{\rm R})^{k+1}$ and its gradient vanish on $[\texttt v_1,\texttt v_c]$ and $[\texttt v_2,\texttt v_c]$ for $k\geq1$. Across $[\texttt v_0,\texttt v_c]$, this factor is continuous, so its tangential derivative agrees from both sides, but its normal derivative may jump. Multiplying by the global linear polynomial $\lambda_2-\lambda_1$, which vanishes on this edge, removes that jump. This motivates the following potential.

\begin{lemma}[Local Airy potential of $\boldsymbol{\psi}^k_0$]
\label{lem:characterization2d}
Let $T$ be a triangle with barycentric refinement $T^{\rm R}$ and barycenter $\texttt v_c$. Let $T_1$ and $T_2$ be the two subtriangles adjacent to the interior edge $[\texttt v_0,\texttt v_c]$, that is,
$T_1\cap T_2=[\texttt v_0,\texttt v_c]$. For an integer $k\ge 1$, define
\begin{equation}\label{eq:v0def}
v_0=\frac{4|T|^2}{9(k+1)}(\lambda_0^{\rm R})^{k+1}(\lambda_2-\lambda_1).
\end{equation}
Then $v_0\in C^1(T)$. Moreover,
\[
J(v_0)=\boldsymbol{\psi}^k_0,
\]
where $\boldsymbol{\psi}^k_0$ is given in \eqref{eq:2d}. Consequently,
$\boldsymbol{\psi}^k_0\in H(\div,T;\mathbb S)$ and
$\div\boldsymbol{\psi}^k_0=0$ on $T$.
\end{lemma}

\begin{proof}
Set $C_T=4|T|^2/9$. The function $v_0$ is continuous and supported on $T_1\cup T_2$. Differentiating \eqref{eq:v0def} on each subtriangle gives
\[
\nabla v_0
=C_T\left[
(\lambda_0^{\rm R})^k(\lambda_2-\lambda_1)\nabla\lambda_0^{\rm R}
+\frac{1}{k+1}(\lambda_0^{\rm R})^{k+1}(\nabla\lambda_2-\nabla\lambda_1)
\right].
\]
On $[\texttt v_1,\texttt v_c]$ and $[\texttt v_2,\texttt v_c]$, both gradient traces vanish because $\lambda_0^{\rm R}=0$ and $k\geq1$. On $[\texttt v_0,\texttt v_c]$, the first term vanishes because $\lambda_2-\lambda_1=0$, and the second term has the same trace from $T_1$ and $T_2$. Thus $\nabla v_0$ is continuous across all interior edges, and $v_0\in C^1(T)$.

Since $\lambda_0^{\rm R}$ is affine on each subtriangle and $\lambda_2-\lambda_1$ is globally affine, a second differentiation gives
\[
\begin{aligned}
\nabla^2v_0=C_T\Bigl[
&k(\lambda_0^{\rm R})^{k-1}(\lambda_2-\lambda_1)
\nabla\lambda_0^{\rm R}\otimes\nabla\lambda_0^{\rm R}\\
&+2(\lambda_0^{\rm R})^k\sym\bigl(\nabla\lambda_0^{\rm R}\otimes(\nabla\lambda_2-\nabla\lambda_1)\bigr)
\Bigr]
\end{aligned}
\]
on each subtriangle. With $\nabla^\bot=(\partial_y,-\partial_x)^{\intercal}$, applying the rotation in \eqref{eq:Jdef} gives
\[
\begin{aligned}
J(v_0)=C_T\Bigl[
&k(\lambda_0^{\rm R})^{k-1}(\lambda_2-\lambda_1)
\nabla^\bot\lambda_0^{\rm R}\otimes\nabla^\bot\lambda_0^{\rm R}\\
&+2(\lambda_0^{\rm R})^k\sym\bigl(\nabla^\bot\lambda_0^{\rm R}\otimes\nabla^\bot(\lambda_2-\lambda_1)\bigr)
\Bigr].
\end{aligned}
\]
The counterclockwise vertex ordering gives
\begin{equation}\label{eq:grad=trot}
\begin{aligned}
\nabla^\bot(\lambda_2-\lambda_1)&=-\frac{3\boldsymbol t_{c,0}}{2|T|},
\quad 
\nabla^\bot(\lambda_0^{\rm R}|_{T_2})=-\frac{3\boldsymbol t_{c,1}}{2|T|},
\quad
\nabla^\bot(\lambda_0^{\rm R}|_{T_1})=\frac{3\boldsymbol t_{c,2}}{2|T|}.
\end{aligned}
\end{equation}
Substituting \eqref{eq:grad=trot}, using $C_T\,9/(4|T|^2)=1$ and
\[
\lambda_2-\lambda_1=-\lambda_1^{\rm R}\quad\text{on }T_2,
\qquad
\lambda_2-\lambda_1=\lambda_2^{\rm R}\quad\text{on }T_1,
\]
recovers \eqref{eq:2d}; both sides vanish on $T_0$. Hence $J(v_0)=\boldsymbol\psi_0^k$. Finally, $v_0\in H^2(T)$ and $\div J(v_0)=0$ imply $\boldsymbol\psi_0^k\in H(\div,T;\mathbb S)$ with zero divergence.
\end{proof}

\subsection{Higher-order $C^1$ elements}
The enrichment functions $\{v_0,v_1,v_2\}$ allow us to construct a $C^1$ element without the vertex $C^2$ smoothness required by the Argyris element \cite{ArgyrisFriedScharpf1968}. In the lattice decomposition~\cite{ChenHuang2021Cmgeodecomp}, the red points in Fig.~\ref{fig:HCTenrichment}(B) belong to two edge layers, preventing their assignment to a single edge. The Argyris element resolves this by extending the vertex layers to include the red lattice points and assigning them to second-derivative vertex DoFs; see Fig.~\ref{fig:HCTenrichment}(A). Each enrichment $v_i$ replaces one red lattice point by two blue copies in Fig.~\ref{fig:HCTenrichment}(C), which can be assigned to different edges.

\begin{figure}[htbp]
\centering
\subfloat[Argyris element $U_{8}^{\rm Arg}(T)$]{
\begin{minipage}[t]{0.3\linewidth}
\centering
\includegraphics[width=0.85\linewidth]{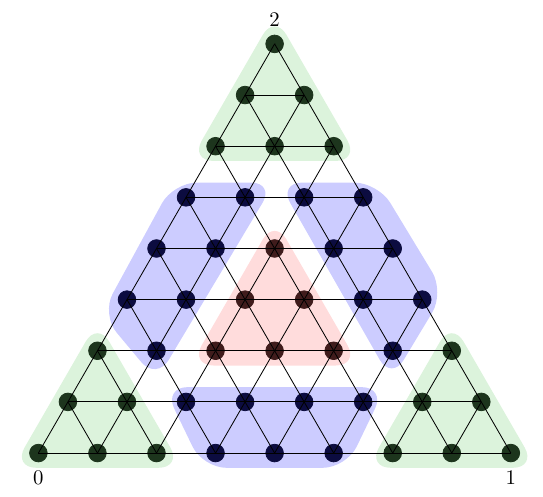}
\end{minipage}
}\hfill
\subfloat[Overlapping edge layers]{
\begin{minipage}[t]{0.3\linewidth}
\centering
\includegraphics[width=0.85\linewidth]{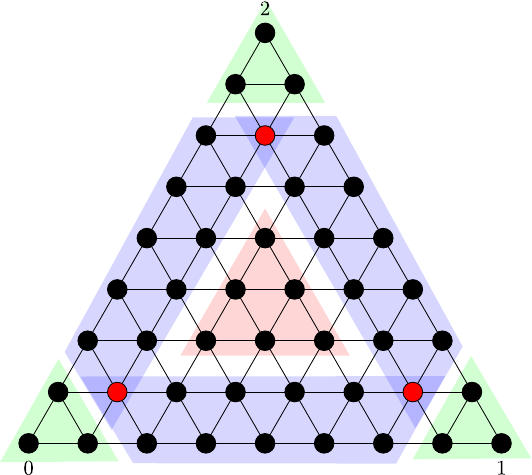}
\end{minipage}
}\hfill
\subfloat[Element $U_{8}(T)$]{
\begin{minipage}[t]{0.3\linewidth}
\centering
\includegraphics[width=0.85\linewidth]{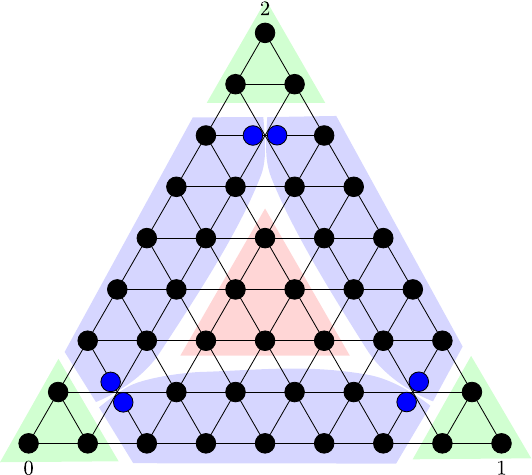}
\end{minipage}
}
\caption{The red lattice points in (B) are duplicated in (C). The copies can be assigned to different edges, which gives enough degrees of freedom for the normal derivative.}
\label{fig:HCTenrichment}
\end{figure}

Let
$$
U_{k+2}(T)=\mathbb{P}_{k+2}(T)\oplus {\rm span}\{v_0,v_1,v_2\}, \qquad k\ge 2.
$$
To construct basis functions adapted to the edge normal-derivative DoFs, we combine the enrichments with the rescaled Bernstein polynomials associated with the red lattice points,
\[
\lambda_0^k\lambda_1\lambda_2,\qquad
\lambda_1^k\lambda_0\lambda_2,\qquad
\lambda_2^k\lambda_0\lambda_1.
\]

Recall that $C_T=4|T|^2/9$ and set $c_{i,i}=\nabla\lambda_i\cdot\boldsymbol n_i$. For $k\geq1$, define the shifted potentials
\begin{equation}\label{eq:shifted-potential}
\widetilde v_i:=\frac{C_T}{k+1}
\bigl[(\lambda_i^{\rm R})^{k+1}-\lambda_i^{k+1}\bigr]
(\lambda_{i-1}-\lambda_{i+1}),
\qquad i=0,1,2\pmod3.
\end{equation}
Each $\widetilde v_i$ differs from $v_i$ by a polynomial and therefore belongs to $C^1(T)$. Since $\lambda_i^{\rm R}=\lambda_i$ on $\partial T$, each shifted potential $\widetilde v_i$ vanishes there.

\begin{lemma}\label{lem:shifted-potential-traces}
For $k\geq1$, the shifted potentials satisfy
\[
\widetilde v_i|_{\partial T}=0,\qquad
\nabla\widetilde v_i|_{\Delta_0(T)}=0.
\]
For $i=0$, their normal traces are
\begin{equation}\label{eq:shifted-potential-traces}
\left.\partial_n\widetilde v_0\right|_{e_0}=0,\qquad
\left.\partial_n\widetilde v_0\right|_{e_1}=-C_Tc_{1,1}\lambda_0^k\lambda_2,\qquad
\left.\partial_n\widetilde v_0\right|_{e_2}=C_Tc_{2,2}\lambda_0^k\lambda_1.
\end{equation}
The remaining traces follow by cyclic permutation.
\end{lemma}
\begin{proof}
Since $\lambda_i^{\rm R}=\lambda_i$ on $\partial T$, each shifted potential vanishes there. Its tangential derivatives along the two incident edges therefore vanish at every vertex. These directions are independent, so its vertex gradient is zero.

We verify the normal traces for $i=0$. Differentiating on each subtriangle and using $\lambda_0^{\rm R}=\lambda_0$ on $e_j$ gives
\[
\left.\nabla\widetilde v_0\right|_{e_j}
=\left.C_T\lambda_0^k(\lambda_2-\lambda_1)
(\nabla\lambda_0^{\rm R}-\nabla\lambda_0)\right|_{e_j}.
\]
On $e_0$, this vanishes because $\lambda_0=0$ and $k\geq1$. On $e_1$ and $e_2$, use
$(\nabla\lambda_0^{\rm R}-\nabla\lambda_0)|_{T_j}=-\nabla\lambda_j$ from \eqref{eq:lambdaRprop1}. Taking the normal component gives \eqref{eq:shifted-potential-traces}.
\end{proof}

For $k\geq2$, define
\[
v_{i,i+1}:=\widetilde v_i-C_T\lambda_i^k\lambda_{i+1}\lambda_{i-1},
\qquad
v_{i,i-1}:=\widetilde v_i+C_T\lambda_i^k\lambda_{i+1}\lambda_{i-1}.
\]
These functions belong to $U_{k+2}(T)$. The bubble correction cancels the normal trace on one adjacent edge, as the following lemma shows. 

\begin{lemma}
For $k\geq 2$, the functions $v_{0,1}$ and $v_{0,2}$,
\[
\begin{aligned}
v_{0,1}
   &= \frac{C_T}{k+1}\bigl[(\lambda_0^{\rm R})^{k+1}-\lambda_0^{k+1}\bigr](\lambda_2-\lambda_1)
   - C_T\lambda_0^k\lambda_1\lambda_2,\\
v_{0,2}
   &= \frac{C_T}{k+1}\bigl[(\lambda_0^{\rm R})^{k+1}-\lambda_0^{k+1}\bigr](\lambda_2-\lambda_1)
   + C_T\lambda_0^k\lambda_1\lambda_2,
\end{aligned}
\]
satisfy
\[
\begin{aligned}
v_{0,1}\big|_{\partial T} &= 0,\qquad
\nabla v_{0,1}\big|_{\Delta_0(T)} = 0,\qquad
\frac{\partial v_{0,1}}{\partial n}\bigg|_{e_i}
= -\frac{8|T|^2}{9}(\nabla\lambda_1\cdot\boldsymbol{n}_{1})\,\lambda_0^k\lambda_2\,\delta_{1,i},\\[4pt]
v_{0,2}\big|_{\partial T} &= 0,\qquad
\nabla v_{0,2}\big|_{\Delta_0(T)} = 0,\qquad
\frac{\partial v_{0,2}}{\partial n}\bigg|_{e_i}
= \frac{8|T|^2}{9}(\nabla\lambda_2\cdot\boldsymbol{n}_{2})\,\lambda_0^k\lambda_1\,\delta_{2,i},
\end{aligned}
\]
where $\delta_{i,j}$ denotes the Kronecker delta for $i, j = 0,1,2$. 
\end{lemma}
\begin{proof}
Set $b=C_T\lambda_0^k\lambda_1\lambda_2$, so $v_{0,1}=\widetilde v_0-b$ and $v_{0,2}=\widetilde v_0+b$. Both $b$ and $\widetilde v_0$ have zero boundary values and vertex gradients. Since $k\geq2$,
\[
\left.\partial_n b\right|_{e_0}=0,\qquad
\left.\partial_n b\right|_{e_1}=C_Tc_{1,1}\lambda_0^k\lambda_2,\qquad
\left.\partial_n b\right|_{e_2}=C_Tc_{2,2}\lambda_0^k\lambda_1.
\]
Combining these identities with \eqref{eq:shifted-potential-traces} proves the stated formulas.
\end{proof}

The remaining trace identities follow by cyclic permutation. Moreover,
$$
\begin{aligned}
\lambda_0^k\lambda_1\lambda_2 &= \frac{9}{8|T|^2}\bigl(v_{0,2}-v_{0,1}\bigr),\\
v_{0} &= \frac{1}{2}\bigl(v_{0,1}+ v_{0,2}\bigr)
+\frac{4|T|^2}{9(k+1)}\lambda_0^{k+1}(\lambda_2-\lambda_1).
\end{aligned}
$$
Hence, when constructing a basis of $U_{k+2}(T)$, we may replace the set
$$
\{\lambda_0^k\lambda_1\lambda_2,\lambda_1^k\lambda_0\lambda_2,\lambda_2^k\lambda_0\lambda_1, v_0, v_1, v_2\}
$$
by the set
$$
\{v_{0,1}, v_{0,2}, v_{1,2}, v_{1,0}, v_{2,0}, v_{2,1}\}.
$$

Following the lattice decomposition in Fig.~\ref{fig:HCTenrichment}, we group the basis functions by subsimplex and by the layers associated with function values and first derivatives. This gives the subspaces
\begin{subequations}
\label{eq:HCTdecomposition}
\begin{align}
U_{\tt v}^0 &= \operatorname{span}\{\lambda_{\tt v}^{k+2}\},
&\quad &{\tt v}\in \Delta_0(T),\\[2pt]
U_{\tt v}^1 &= \operatorname{span}\{\lambda_{\tt v}^{k+1}\lambda_{{\tt v}^*(0)},
\ \lambda_{\tt v}^{k+1}\lambda_{{\tt v}^*(1)}\},
&\quad &{\tt v}\in \Delta_0(T),\\
U_e^0 &= b_e^2\,\mathbb P_{k-2}(e),
&\quad &e\in \Delta_1(T),\\[2pt]
U_e^1 &= b_e b_T\,\mathbb P_{k-3}(e)
\begin{aligned}[t]
&\oplus\ \operatorname{span}\{v_{e(0),e^*},\ v_{e(1),e^*}\},
\end{aligned}
&\quad &e\in \Delta_1(T),\\[2pt]
U_T^0 &= b_T^2\,\mathbb P_{k-4}(T).
\end{align}
\end{subequations}
Define
$$
U_{\Delta_\ell(T)}^m := \Oplus_{f\in \Delta_\ell(T)} U_f^m,
\qquad \ell=0,1,\quad m=0,1,
$$
where $\ell$ is the dimension of the subsimplex $f$, and $m$ indicates the order of the derivative.

Let $\tilde U_e^1 := b_e b_T\,\mathbb P_{k-3}(e)$ and
$$
U^{\rm red} := \operatorname{span}\{\lambda_0^k\lambda_1\lambda_2, \lambda_1^k\lambda_0\lambda_2, \lambda_2^k\lambda_0\lambda_1\}.
$$
The geometric decomposition of the simplicial lattice, illustrated in Fig.~\ref{fig:HCTenrichment}(B), gives
$$
\mathbb P_{k+2}(T) = \Oplus_{{\tt v}\in\Delta_0(T)}(U_{\tt v}^0 \oplus U_{\tt v}^1) \oplus \Oplus_{e\in\Delta_1(T)}(U_e^0 \oplus \tilde U_e^1)\oplus U^{\rm red} \oplus U_T^0.
$$
Using the basis replacement above, we obtain the following decomposition of $U_{k+2}(T)$.

\begin{lemma}
The following direct-sum decomposition holds:
\begin{equation}
\label{eq:HCTdec}
U_{k+2}(T) = U_T^0\oplus\Oplus_{\ell=0}^1\Oplus_{m=0}^1 U_{\Delta_\ell(T)}^m.
\end{equation}
\end{lemma}

The decomposition gives the local degrees of freedom.

\begin{theorem}\label{th:Uk+2}
For $k\ge 2$, a function $v\in U_{k+2}(T)$ is uniquely determined by the following degrees of freedom:
\begin{subequations}\label{HCTfeDoF}
\begin{align}
v(\mathtt v),
&\quad  \mathtt v \in \Delta_0(T), \label{HCTfeDoF1}\\
\nabla v(\mathtt v),
&\quad  \mathtt v \in \Delta_0(T), \label{HCTfeDoF2}\\
\int_e v\, q \,\dd s,
&\quad  q \in \mathbb P_{k-2}(e),\ e \in \Delta_1(T), \label{HCTfeDoF3}\\
\int_e \frac{\partial v}{\partial n}\, q \,\dd s,
&\quad  q \in \mathbb P_{k-1}(e),\ e \in \Delta_1(T), \label{HCTfeDoF4}\\
\int_T v\, q \,\dd x,
&\quad  q \in \mathbb P_{k-4}(T). \label{HCTfeDoF5}
\end{align}
\end{subequations}
\end{theorem}

\begin{proof}
The total number of DoFs in \eqref{HCTfeDoF} is
$$
3 \;+\; 6 \;+\; 3(k-1) \;+\; 3k \;+\; \frac{(k-2)(k-3)}{2}
= \frac{(k+4)(k+3)}{2}+3,
$$
which agrees with $\dim U_{k+2}(T)$.

Let $D_{\Delta_0(T)}^0$, $D_{\Delta_0(T)}^1$, $D_{\Delta_1(T)}^0$, $D_{\Delta_1(T)}^1$, and $D_T^0$ denote the collections of DoFs in \eqref{HCTfeDoF1}--\eqref{HCTfeDoF5}, respectively. For $U_e^1$, by the definition of $v_{i,j}$,
\begin{equation}\label{eq:uepartial_ndof}
\frac{\partial}{\partial n}\,U_e^1\big|_{e'}
= b_e\,\mathbb P_{k-1}(e)\,\delta_{e,e'},
\qquad \forall\, e,e'\in\Delta_1(T).
\end{equation}
Both $U_e^1$ and $b_e\mathbb P_{k-1}(e)$ have dimension $k$. Hence the surjective normal-trace map in \eqref{eq:uepartial_ndof} is an isomorphism from $U_e^1$ onto $b_e\mathbb P_{k-1}(e)$.

Applying the functionals $D_{\Delta_0(T)}^0,\dots,D_T^0$ to each subspace in \eqref{eq:HCTdecomposition}, and using the standard properties of Bernstein polynomials together with \eqref{eq:uepartial_ndof}, we obtain the following block pattern:
\begin{equation}\label{eq:lowertriangular-new}
\renewcommand{\arraystretch}{1.25}
\begin{array}{c|*{5}{c}}
\displaystyle D \;\backslash\; U
& U_{\Delta_0(T)}^0
& U_{\Delta_0(T)}^1
& U_{\Delta_1(T)}^0
& U_{\Delta_1(T)}^1
& U_T^0 \\ \hline
D_{\Delta_0(T)}^0
& \graysquare & 0 & 0 & 0 & 0 \\
D_{\Delta_0(T)}^1
& \graysquare & \graysquare & 0 & 0 & 0 \\
D_{\Delta_1(T)}^0
& \graysquare & \graysquare & \graysquare & 0 & 0 \\
D_{\Delta_1(T)}^1
& \graysquare & \graysquare & \graysquare & \graysquare & 0 \\
D_T^0
& \graysquare & \graysquare & \graysquare & \graysquare & \graysquare
\end{array}
\end{equation}
The DoF-basis matrix is block lower triangular, with diagonal blocks
$$
\langle D_f^m, U_f^m \rangle,\quad f\in \Delta_\ell(T),\ \ell=0,1,\ m=0,1,
\qquad \langle D_T^0,U_T^0\rangle.
$$
The vertex blocks are invertible by evaluation of the vertex values and gradients. The edge-value and interior blocks are weighted polynomial mass matrices with weights $b_e^2$ and $b_T^2$, respectively. By \eqref{eq:uepartial_ndof}, identifying $U_e^1$ with its normal trace reduces the edge normal-derivative block to a weighted polynomial mass matrix with weight $b_e$. These weights are positive in the relative interiors of their respective edges or triangle, so all diagonal blocks are invertible. Hence the full matrix is invertible, proving unisolvence.
\end{proof}

\begin{remark}
A dual basis can be computed by inverting the DoF-basis matrix~\cite{ChenChenWeiothers2023Geometric,ChenChenGaoHuangWei2025Basis}. The functions $v_{i,i\pm1}$ are constructed to annihilate the vertex and edge-value DoFs \eqref{HCTfeDoF1}--\eqref{HCTfeDoF3}, giving the zero blocks above the diagonal in their columns.
\end{remark}

Fix a unit normal $\boldsymbol n_e$ on each mesh edge $e$ and use this same normal for the normal-derivative DoFs from both adjacent triangles.
We define the global finite element space on the triangulation $\mathcal T_h$ by
$$
\begin{aligned}
U_{k+2,h}
:= \{v\in L^2(\Omega):\ & v|_T\in U_{k+2}(T),\ \forall\, T\in\mathcal T_h,\\
&\text{and the DoFs \eqref{HCTfeDoF1}--\eqref{HCTfeDoF4} are single-valued}\}.
\end{aligned}
$$
For $e\in\Delta_1(T)$, we have $U_{k+2}(T)|_e\subseteq \mathbb P_{k+2}(e)$ and $\partial_nU_{k+2}(T)|_e\subseteq \mathbb P_{k+1}(e)$. Since $\mathbb P_{k+2}(e)$ is determined by the restricted DoFs \eqref{HCTfeDoF1}--\eqref{HCTfeDoF3}, and $\mathbb P_{k+1}(e)$ is determined by the restricted DoFs \eqref{HCTfeDoF2} and \eqref{HCTfeDoF4}, it follows that $U_{k+2,h}\subset C^1(\Omega)$.

\begin{theorem}[Dimension minimality of the $C^1$ macroelement]\label{th:minimal-C1-macroelement}
Let $q\geq4$, and let $V_q(T)$ be a local $C^1$ macroelement space on an internal refinement of a triangle $T$. Assume that:
\begin{enumerate}
\item $\mathbb P_q(T)\subset V_q(T)$;
\item the global assembly imposes only single-valued vertex values, vertex gradients, and edge traces, with no vertex supersmoothness;
\item the joint boundary trace map
$$
\gamma_{\partial T}:V_q(T)\longrightarrow\mathcal B_q(\partial T),
\qquad u\longmapsto (u|_e,\partial_nu|_e)_{e\subset\partial T},
$$
is onto, where $\mathcal B_q(\partial T)$ consists of data in $\mathbb P_q(e)\times\mathbb P_{q-1}(e)$ on all three edges, compatible only with common vertex values and gradients.
\end{enumerate}
Then any assembled global space $V_{q,h}\subset C^1(\Omega)$ satisfies
$$
\dim V_{q,h}\geq
3|\Delta_0(\mathcal T_h)|
+
(2q-5)|\Delta_1(\mathcal T_h)|
+
\binom{q-4}{2}|\Delta_2(\mathcal T_h)|.
$$
The space $U_{q,h}$ constructed above attains this bound. Hence $U_{q,h}$ has minimal dimension in this class.
\end{theorem}

\begin{proof}
Joint trace surjectivity allows all compatible boundary data to be chosen independently. The vertex values and gradients contribute three global DoFs per vertex. On an edge $e$, the endpoint vertex data determine the values and tangential derivatives of $u|_e$ at both endpoints. Since $\dim\mathbb P_q(e)=q+1$, at least $q-3$ additional edge DoFs are required to represent all compatible traces of $u$. Similarly, the endpoint gradients determine the two endpoint values of $\partial_nu|_e$. Since $\dim\mathbb P_{q-1}(e)=q$, at least $q-2$ additional edge DoFs are required to represent all compatible normal derivative traces. Thus each edge contributes at least $2q-5$ DoFs.

Moreover, since $\mathbb P_q(T)\subset V_q(T)$, the subspace of polynomial bubbles with vanishing value and normal derivative on $\partial T$, $b_T^2\mathbb P_{q-6}(T),$
is contained in the kernel of the boundary trace map. Hence each element contributes at least
$\dim\mathbb P_{q-6}(T)=\binom{q-4}{2}$
interior DoFs. Summing these lower bounds over vertices, edges, and elements proves the result.

The space $U_{q,h}$ has exactly three DoFs per vertex, $2q-5$ per edge, and $\binom{q-4}{2}$ per element. Therefore $U_{q,h}$ attains the lower bound.
\end{proof}

\subsection{Low-order $C^1$ elements}\label{sec:loworder}

The low-order constructions recover the classical HCT and reduced HCT elements~\cite{CloughTocher1965,ChristHu2018Generalized}. We retain explicit formulas connecting these spaces to our enrichment functions and supporting the dual bases constructed below.

Specializing the formulas for $v_{i,j}$ to $k=1$ ($q=3$), the six functions $v_{i,i\pm1}$ are linearly dependent. Indeed, recalling that $b_T=\lambda_0\lambda_1\lambda_2$, direct calculation gives
\begin{equation}\label{eq:vdependent}
v_{0,2}-v_{0,1}
=
v_{1,0}-v_{1,2}
=
v_{2,1}-v_{2,0}
=
2C_T\,b_T.
\end{equation}
The cubic simplicial lattice $\mathbb T^2_3$ has only one interior lattice point, corresponding to $b_T$. 

For $k=1$, use the shifted potentials \eqref{eq:shifted-potential}, whose boundary values and vertex gradients vanish by Lemma~\ref{lem:shifted-potential-traces}. Define
\begin{equation}\label{eq:wbasis}
w_i:=\frac{1}{3C_Tc_{i,i}} (\widetilde v_{i+1}-\widetilde v_{i-1}+C_Tb_T),
\qquad i=0,1,2\pmod3.
\end{equation}

\begin{lemma}
The functions $w_i$ defined in \eqref{eq:wbasis} satisfy
\[
w_i|_{\partial T}=0,\qquad
\nabla w_i|_{\Delta_0(T)}=0,\qquad
\left.\partial_{n}w_i\right|_{e_j}=\delta_{ij}b_{e_j},
\qquad i,j=0,1,2.
\]

\end{lemma}

\begin{proof}
The boundary values and vertex gradients vanish by Lemma~\ref{lem:shifted-potential-traces} and the corresponding properties of $b_T$. For the normal traces, \eqref{eq:shifted-potential-traces} with $k=1$ and cyclic permutation give
\[
\left.\partial_n(\widetilde v_1-\widetilde v_2)\right|_{e_j}
=C_Tc_{j,j}b_{e_j}
\begin{cases}
2,&j=0,\\
-1,&j=1,2.
\end{cases}
\]
Since $\partial_n b_T|_{e_j}=c_{j,j}b_{e_j}$, adding $C_Tb_T$ cancels the traces on $e_1,e_2$ and gives $3C_Tc_{0,0}b_{e_0}$ on $e_0$. Division by $3C_Tc_{0,0}$ proves the assertion for $w_0$; the other cases follow by cyclic permutation.
\end{proof}

Using the vertex subspaces $U_{\mathtt v}^0,U_{\mathtt v}^1$ in \eqref{eq:HCTdecomposition} with $k=1$, we obtain the following result.

\begin{proposition}\label{cor:U3}
Let
\[
U_3(T):=\mathbb P_3(T)+\operatorname{span}\{v_0,v_1,v_2\}
=\mathbb P_3(T)+\operatorname{span}\{w_0,w_1,w_2\}.
\]
Then
\[
U_3(T)=\Oplus_{{\tt v}\in \Delta_0(T)}(U_{\tt v}^0\oplus U_{\tt v}^1)\oplus \operatorname{span}\{w_0,w_1,w_2\}.
\]
In particular, $\dim U_3(T)=12$, and $U_3(T)$ is the classical HCT space. Its degrees of freedom are
\[
v(\mathtt v_i),\quad \nabla v(\mathtt v_i),\quad
\int_{e_i}\partial_n v\,\dd s,\qquad i=0,1,2.
\]

\end{proposition}

\begin{proof}
The nine vertex polynomials in $U_{\mathtt v}^0$ and $U_{\mathtt v}^1$ are determined by their vertex values and gradients. The functions $w_i$ have zero vertex DoFs and satisfy
\[
\int_{e_j}\partial_n w_i\,\dd s=\delta_{ij}\int_{e_j}b_{e_j}\,\dd s=\delta_{ij}|e_i|/6.
\]
Hence these twelve functions are linearly independent, and the stated DoFs are unisolvent on their span.

By \eqref{eq:wbasis}, their span is contained in $\mathbb P_3(T)+\operatorname{span}\{v_0,v_1,v_2\}$, which is contained in the classical HCT space $C^1(T)\cap\mathbb P_3(T^{\rm R})$. Since the latter has dimension $12$~\cite{CloughTocher1965,ChristHu2018Generalized}, all these spaces coincide. This proves both representations of $U_3(T)$, the direct-sum decomposition, and unisolvence.
\end{proof}

The classical reduced HCT space~\cite{Ciarlet1978ReducedHCT} is
\[
U_2(T)
:=
\bigl\{
v\in U_3(T):
\ \partial_n v|_{e}\in \mathbb P_1(e),
\ \forall\, e\in\Delta_1(T)
\bigr\}.
\]
It contains $\mathbb P_2(T)$ and has dimension $12-3=9$, since the $w_i$ independently control the three quadratic normal-trace coefficients. Its nine vertex DoFs are unisolvent: zero vertex data imply zero linear normal traces, so all HCT DoFs vanish and Proposition~\ref{cor:U3} applies.

To construct the three enrichments, observe that a quadratic normal trace differs from its linear endpoint interpolant by a multiple of $b_e$. Subtracting a suitable multiple of $w_i$ removes this difference on $e_i$ without changing the normal traces on the other edges, the function trace, or the vertex DoFs. Set $c_{i,j}:=\nabla\lambda_i\cdot\boldsymbol n_j$.

\begin{lemma}
For $i=0,1,2$, with indices understood modulo $3$, define
\begin{equation}\label{eq:u0def}
\begin{aligned}
u_i:={}&\lambda_{i+1}\lambda_{i-1}(\lambda_{i+1}-\lambda_{i-1})
-3(c_{i+1,i}-c_{i-1,i})w_i\\
&-c_{i+1,i+1}w_{i+1}+c_{i-1,i-1}w_{i-1}.
\end{aligned}
\end{equation}
Then $u_i\in U_2(T)$ and
\[
U_2(T)=\mathbb P_2(T)\oplus\operatorname{span}\{u_0,u_1,u_2\}.
\]
Thus $\dim U_2(T)=9$, and its degrees of freedom are the vertex values and gradients
\[
v(\mathtt v_i),\quad \nabla v(\mathtt v_i),\qquad i=0,1,2.
\]
\end{lemma}

\begin{proof}
For $i=0$, differentiation and $\partial_n w_i|_{e_j}=\delta_{ij}b_{e_j}$ give
\[
\left.\partial_n u_0\right|_{e_0}=c_{2,0}\lambda_1-c_{1,0}\lambda_2,\qquad
\left.\partial_n u_0\right|_{e_1}=-c_{1,1}\lambda_2,\qquad
\left.\partial_n u_0\right|_{e_2}=c_{2,2}\lambda_1.
\]
Thus $u_0\in U_2(T)$; the other cases follow by cyclic permutation.

Finally, $w_j|_{\partial T}=0$ implies
\[
u_j|_{e_i}=\delta_{ij}\lambda_{i+1}\lambda_{i-1}(\lambda_{i+1}-\lambda_{i-1}).
\]
Each $u_i$ has a nonzero cubic trace on $e_i$ and zero trace on the other edges. Hence the three functions are linearly independent modulo $\mathbb P_2(T)$. Since $\dim U_2(T)=9=6+3$, the direct sum equals $U_2(T)$.
\end{proof}

Fix one unit normal $\boldsymbol n_e$ on each mesh edge $e$ and use the same normal for both adjacent triangles.
For a triangulation $\mathcal T_h$, we define the global spaces
\[
\begin{aligned}
U_{3,h}
&:=
\{v\in L^2(\Omega):\ v|_T\in U_3(T)\ \forall\,T\in\mathcal T_h,
\text{ with single-valued DoFs}\},\\
U_{2,h}
&:=
\{v\in L^2(\Omega):\ v|_T\in U_2(T)\ \forall\,T\in\mathcal T_h,
\text{ with single-valued vertex DoFs}\}.
\end{aligned}
\]

\begin{proposition}
For $k=0,1$, the space $U_{k+2,h}$ is a subspace of $C^1(\Omega)$.
\end{proposition}

\begin{proof}
On each shared edge, the single-valued vertex values and gradients determine the same cubic function trace from both adjacent triangles. Use the same fixed normal on both sides. For $U_{3,h}$, the endpoint gradients and the edge moment uniquely determine the quadratic normal trace. For $U_{2,h}$, the endpoint gradients determine the linear normal trace. Thus both the function and its gradient agree across every mesh edge.
\end{proof}

\subsection{Bases dual to degrees of freedom}\label{sec:dual-bases}
For completeness, we give explicit bases for $U_2$ and $U_3$ dual to slightly modified DoFs.

For $e\in\Delta_1(T)$, define a linear functional $D_e:C(e)\to\mathbb{R}$ by
\[
D_e(v):=\frac{6}{|e|}\int_e v\,\dd s - 2\bigl(v(e(0))+v(e(1))\bigr).
\]
If $v\in\mathbb P_2(e)$ is written in Bernstein form as
$v=a\lambda_{e(0)}^2+b\lambda_{e(0)}\lambda_{e(1)}+c\lambda_{e(1)}^2$,
then $D_e(v)=b$. For the normalized Bernstein basis, $B^{(1,1)}=2\lambda_{e(0)}\lambda_{e(1)}$. Thus $D_e|_{\mathbb P_2(e)}=2\texttt{b}^{(1,1)}$, where $\texttt{b}^{(1,1)}$ is the corresponding dual-Bernstein functional used in the $C^m$ basis construction \cite{ChenChenGaoHuangWei2025Basis}.
We replace the edge moment DoF \eqref{HCTfeDoF4} for $U_3(T)$ by
\begin{equation}\label{HCTfeDoFpn}
D_e(\partial_n v|_e), \quad e\in\Delta_1(T).
\end{equation}
Since $\partial_n v|_e\in\mathbb P_2(e)$ and its endpoint values are already determined by the vertex gradient DoFs, the modified DoFs \eqref{HCTfeDoF1}--\eqref{HCTfeDoF2} and \eqref{HCTfeDoFpn} remain unisolvent for $U_3(T)$ and preserve $C^1$ conformity.

The directional-derivative basis below is local to $T$. An element-dependent change of basis converts it to the Cartesian derivative basis used for global assembly.
For $i=0,1,2$, with indices understood modulo $3$, we also replace the vertex gradient DoF \eqref{HCTfeDoF2} by the directional derivatives
\begin{equation}\label{HCTfeDoFpt}
\frac{\partial v}{\partial \boldsymbol t_{i,i-1}}(\mathtt v_i),\quad
\frac{\partial v}{\partial \boldsymbol t_{i,i+1}}(\mathtt v_i).
\end{equation}
Let
\[
F_{i,T}=\begin{pmatrix}\boldsymbol t_{i,i-1} & \boldsymbol t_{i,i+1}\end{pmatrix}.
\]
Then
\[
\Bigl(
\frac{\partial v}{\partial \boldsymbol t_{i,i-1}}(\mathtt v_i),\,
\frac{\partial v}{\partial \boldsymbol t_{i,i+1}}(\mathtt v_i)
\Bigr)
= \nabla v(\mathtt v_i)^{\intercal} F_{i,T}.
\]
If $\phi_{i,i-1}$ and $\phi_{i,i+1}$ are dual to the DoFs \eqref{HCTfeDoFpt}, then the dual basis for \eqref{HCTfeDoF2} is given by
\[
(\psi_{i,1},\psi_{i,2})
=(\phi_{i,i-1},\phi_{i,i+1})\, F_{i,T}^{\intercal},
\qquad i=0,1,2.
\]
Hence it is enough to construct bases for $U_3(T)$ dual to \eqref{HCTfeDoF1}, \eqref{HCTfeDoFpt} and \eqref{HCTfeDoFpn}, and for $U_2(T)$ dual to \eqref{HCTfeDoF1} and \eqref{HCTfeDoFpt}.

\subsubsection{Explicit dual bases for $U_3(T)$}
Recall from \eqref{eq:wbasis} that, for $k=1$,
\[
w_i=\frac{\widetilde v_{i+1}-\widetilde v_{i-1}+C_Tb_T}{3C_Tc_{i,i}},\qquad C_T=\frac{4|T|^2}{9},\quad c_{i,i}=\nabla\lambda_i\cdot\boldsymbol n_i.
\]
These functions have zero vertex DoFs and satisfy $D_{e_j}(\partial_n w_i)=\delta_{ij}$. It remains to correct the vertex basis functions so that their edge DoFs vanish. 
For $i=0,1,2$, with indices understood modulo $3$, define
\[
\phi_{i,i+1}^1 := \lambda_i^2\lambda_{i+1} - 2c_{i,i-1}\,w_{i-1},
\qquad
\phi_{i,i-1}^1 := \lambda_i^2\lambda_{i-1} - 2c_{i,i+1}\,w_{i+1}.
\]
These functions are dual to the directional derivative DoFs \eqref{HCTfeDoFpt}. The dual basis functions associated with the vertex value DoFs are
\[
\phi_i^0 := \lambda_i^3 + 3\phi_{i,i+1}^1 + 3\phi_{i,i-1}^1,
\qquad i=0,1,2.
\]
The twelve local basis functions dual to the vertex values, directional derivatives, and modified edge DoFs are
$$
(\phi_i^0, \phi_{i,i+1}^1, \phi_{i,i-1}^1, w_i), \quad i=0,1,2.
$$

\subsubsection{Explicit dual bases for $U_2(T)$}
Recall from \eqref{eq:u0def} that
\[
\begin{aligned}
u_i={}&\lambda_{i+1}\lambda_{i-1}(\lambda_{i+1}-\lambda_{i-1})
-3(c_{i+1,i}-c_{i-1,i})w_i\\
&-c_{i+1,i+1}w_{i+1}+c_{i-1,i-1}w_{i-1}.
\end{aligned}
\]
Since the $w_j$ have zero vertex gradients, for $i=0,1,2$, with indices understood modulo $3$, we have
\[
\partial_{\boldsymbol t_{i+1,i-1}} u_i(\mathtt v_{i+1})=1,
\qquad
\partial_{\boldsymbol t_{i-1,i+1}} u_i(\mathtt v_{i-1})=-1,
\]
and
\[
\partial_{\boldsymbol t_{i+1,i-1}} (\lambda_{i+1}\lambda_{i-1})(\mathtt v_{i+1})=1,
\qquad
\partial_{\boldsymbol t_{i-1,i+1}} (\lambda_{i+1}\lambda_{i-1})(\mathtt v_{i-1})=1.
\]
Hence the basis functions dual to the directional derivative DoFs are
\[
\phi_{i,i+1}^1 := \tfrac12(\lambda_i\lambda_{i+1} + u_{i-1}),
\qquad
\phi_{i,i-1}^1 := \tfrac12(\lambda_i\lambda_{i-1} - u_{i+1}),
\qquad i=0,1,2.
\]
The basis functions dual to the vertex value DoFs are
\[
\phi_i^0 := \lambda_i^2 + 2\phi_{i,i+1}^1 + 2\phi_{i,i-1}^1,
\qquad i=0,1,2.
\]
The nine local basis functions dual to the vertex values and directional derivatives are
$$
(\phi_i^0, \phi_{i,i+1}^1, \phi_{i,i-1}^1), \quad i=0,1,2.
$$

\subsection{Relation to virtual elements}\label{sec:vem}
For $k\ge1$, the local $C^1$ virtual element space of degree $k+2$ is~\cite{ChenHuangWei2022,BrezziMarini2013,BeiraodaVeigaManzini2014}
\begin{align*}
V_{k+2}^{\rm VEM}(T):=\{v\in H^2(T)&: v|_{\partial T}\in C(\partial T),\ \nabla v|_{\partial T}\in C(\partial T;\mathbb R^2),\ \Delta^2v\in \mathbb P_{k-2}(T), \\
&\;\;\; v|_e\in\mathbb P_{k+2}(e),\ \partial_nv|_e\in\mathbb P_{k+1}(e)\ \textrm{ for } e\in\Delta_1(T) \}.
\end{align*}
The space contains $\mathbb P_{k+2}(T)$. Its functions are uniquely determined by the DoFs \eqref{HCTfeDoF}, with \eqref{HCTfeDoF5} replaced by
\begin{equation}\label{C1veDoF5}
\int_T v\, q \,\dd x,
\quad q \in \mathbb P_{k-2}(T).
\end{equation}

For $k=1$, $\dim V_{3}^{\rm VEM}(T)=\dim U_{3}(T)$. For $k=0$, $\dim U_{2}(T)$ matches the dimension of the constrained virtual space
\begin{equation*}
\{v\in V_{3}^{\rm VEM}(T): \partial_n v|_e\in\mathbb P_{1}(e)\ \textrm{ for } e\in\Delta_1(T)\}.
\end{equation*}
For $q=k+2\geq4$, the dimension difference is
\[
\dim V_q^{\rm VEM}(T)-\dim U_q(T)
=\dim\mathbb P_{q-4}(T)-\dim\mathbb P_{q-6}(T)=2q-7.
\]
Consequently, deleting interior moments from this unchanged virtual space would destroy unisolvence. Reducing the DoFs therefore requires a corresponding modification of the virtual space. The Hessian energy projection involves interior moments through degree $q-4$, whereas a full $L^2$ projection onto $\mathbb P_q$ generally needs additional information or enhancement constraints. The comparison here concerns boundary data and dimension, not equality of spaces. Virtual elements accommodate general polygons and typically require stabilization; every function of $U_q(T)$ is instead explicit on the barycentric refinement.

\section{Finite element elasticity complex in two dimensions}\label{sec:feelascomplex}
We establish commuting interpolation, divergence stability, and cohomology preservation for the higher-order spaces. We then construct bounded commuting maps on the Hilbert complex.

We first consider $k\ge2$.
Set $V_{m,h}:=\{\boldsymbol v\in L^2(\Omega;\mathbb R^2):\boldsymbol v|_T\in\mathbb P_m(T;\mathbb R^2)\ \forall T\in\mathcal T_h\}$ and let $Q_{m,h}$ be the $L^2$ projection onto $V_{m,h}$. The local constructions and conformity give
\[
JU_{k+2,h}\subseteq\Sigma_{k,h},
\qquad
\div\Sigma_{k,h}\subseteq V_{k-1,h}.
\]

Let $I_{k,h}^{\div}: H^1(\Omega;\mathbb S)\to\Sigma_{k,h}$ and $I_{k+2,h}^{H^2}: H^3(\Omega)\to U_{k+2,h}$ be the interpolation operators defined by \eqref{DoF} and \eqref{HCTfeDoF}, respectively. Each operator preserves its defining DoFs:
$$
N(I_hu)=N(u)\qquad\text{ for all DoFs }N.
$$

\begin{lemma}
The interpolation operators satisfy
\begin{align}
\div(I_{k,h}^{\div}\boldsymbol{\tau})
&= Q_{k-1,h}(\div\boldsymbol{\tau}),
&&\forall\,\boldsymbol{\tau}\in H^1(\Omega;\mathbb S),\label{eq:commuprop1}\\
I_{k,h}^{\div}(Jv)
&= J(I_{k+2,h}^{H^2}v),
&&\forall\,v\in H^3(\Omega).\label{eq:commuprop2}
\end{align}
\end{lemma}
\begin{proof}
For $\bs q\in \mathbb P_{k-1}(T;\mathbb R^2)$, integration by parts gives
$$
\begin{aligned}
\int_T \div(I_{k,h}^{\div}\boldsymbol{\tau})\cdot\bs q \,\dd x
&= - \int_T I_{k,h}^{\div}\boldsymbol{\tau} : \nabla \bs q \,\dd x + \int_{\partial T}(I_{k,h}^{\div}\boldsymbol{\tau}\bs n)\cdot \bs q \,\dd s\\
&= - \int_T \boldsymbol{\tau} : \nabla \bs q \,\dd x + \int_{\partial T}(\boldsymbol{\tau}\bs n)\cdot \bs q \,\dd s = \int_T \div\boldsymbol{\tau}\cdot\bs q \,\dd x.
\end{aligned}
$$
The second equality uses the edge and interior moments in \eqref{DoF}, together with the symmetry of the stresses. Since $\div I_{k,h}^{\div}\boldsymbol\tau\in V_{k-1,h}$, this proves \eqref{eq:commuprop1}.

For \eqref{eq:commuprop2}, let $e\in\Delta_1(T)$ and $\bs q\in \mathbb P_k(e;\mathbb R^2)$. Use $\curl z=\nabla^\bot z$ and orient $\boldsymbol t=(-n_y,n_x)$, so that $(Jz)\boldsymbol n=\partial_t\curl z$. Then
\begin{align*}
\int_e \bigl(\bigl(J(v-I_{k+2,h}^{H^2}v)\bigr)\bs n\bigr)\cdot \bs q \,\dd s
= -\int_e \bigl(\curl(v-I_{k+2,h}^{H^2}v)\bigr)\cdot \partial_t\bs q \,\dd s =0.
\end{align*}
The endpoint terms vanish by \eqref{HCTfeDoF2}; the remaining integral vanishes by \eqref{HCTfeDoF1}, \eqref{HCTfeDoF3}, and \eqref{HCTfeDoF4}, after tangential integration by parts. Two integrations by parts over $T$, using \eqref{HCTfeDoF}, also give, for $\bs\tau\in\mathbb P_{k-2}(T;\mathbb S)$,
\begin{equation*}
\int_T J(v-I_{k+2,h}^{H^2}v): \bs \tau \,\dd x = 0.
\end{equation*}
Thus $Jv$ and $J I_{k+2,h}^{H^2}v$ have the same stress DoFs, proving \eqref{eq:commuprop2}.
\end{proof}

\begin{lemma}
\label{lem:infsup}
Let $k\ge 2$. Then
$$
\div \Sigma_{k,h} = V_{k-1,h},
$$
and
\[
\|\boldsymbol{v}\|_{0}
\lesssim
\sup_{0\ne\boldsymbol{\tau}\in\Sigma_{k,h}}
\frac{(\div\boldsymbol{\tau},\boldsymbol{v})}{\|\boldsymbol{\tau}\|_{H(\div)}},
\quad \forall\,\boldsymbol{v}\in V_{k-1,h}.
\]
\end{lemma}
\begin{proof}
For any $\boldsymbol{v}\in V_{k-1,h}$, there exists $\boldsymbol{\tau}\in H^1(\Omega;\mathbb S)$ \cite[Lemma 6]{JohnsonMercier1978} such that
$$
\div \boldsymbol{\tau} = \boldsymbol{v},
\qquad
\|\boldsymbol{\tau}\|_{1}\lesssim \|\boldsymbol{v}\|_{0}.
$$
Scaling the local DoFs from a reference triangle gives, with $h_T=\operatorname{diam}T$,
\begin{equation}\label{eq:stress-interpolation-bound}
\|I_{k,T}^{\div}\boldsymbol\tau\|_{0,T}
\lesssim\|\boldsymbol\tau\|_{0,T}+h_T|\boldsymbol\tau|_{1,T}.
\end{equation}
The constant is uniform on the shape-regular family. Let $\boldsymbol{\tau}_I := I_{k,h}^{\div}\boldsymbol{\tau}$. Summing this bound and using \eqref{eq:commuprop1} gives
\[
\div \boldsymbol{\tau}_I = \div \boldsymbol{\tau} = \boldsymbol{v},
\qquad
\|\boldsymbol{\tau}_I\|_{H(\div)}\lesssim \|\boldsymbol{v}\|_{0}.
\]
This proves both statements.
\end{proof}

These identities give a commuting discrete elasticity complex.

\begin{theorem}
\label{thm:femelasticitycomplexHCT}
Let $k\ge 2$. The following diagram
\begin{equation}\label{eq:femelasticitycomplexHCT}
\begin{array}{c}
\xymatrix{
\mathbb P_1 \ar[r]^-{\subset}
& H^3(\Omega) \ar[d]^{I_{k+2,h}^{H^2}} \ar[r]^-{J}
& H^1(\Omega;\mathbb S) \ar[d]^{I_{k,h}^{\div}} \ar[r]^-{\div}
& L^2(\Omega;\mathbb R^2) \ar[d]^{Q_{k-1,h}} \ar[r]
& 0
\\
\mathbb P_1 \ar[r]^-{\subset}
& U_{k+2,h} \ar[r]^-{J}
& \Sigma_{k,h} \ar[r]^-{\div}
& V_{k-1,h} \ar[r]
& 0
}
\end{array}
\end{equation}
commutes. 
The inclusion of the discrete complex into the Hilbert complex \eqref{eq:2Dcomplex} induces isomorphisms on cohomology. If $\Omega$ has $g$ holes, the stress cohomology has dimension $3g$; in particular, the discrete complex is exact when $\Omega$ is simply connected.
\end{theorem}
\begin{proof}
Commutativity follows from \eqref{eq:commuprop1}--\eqref{eq:commuprop2}. 
Let $N_{\texttt{v}},N_e,N_T$ denote the numbers of macrovertices, macroedges, and macrotriangles, including the boundary. The DoFs give
\begin{equation}\label{eq:global-dimensions}
\begin{aligned}
\dim U_{k+2,h}&=3N_{\texttt{v}}+(2k-1)N_e+\tfrac12(k-2)(k-3)N_T,\\
\dim\Sigma_{k,h}&=2(k+1)N_e+\tfrac32k(k-1)N_T,\qquad
\dim V_{k-1,h}=k(k+1)N_T.
\end{aligned}
\end{equation}
Euler's formula gives the alternating sum $3(N_{\texttt{v}}-N_e+N_T)=3(1-g)$. Since $\ker J=\mathbb P_1$ on the connected domain and $\div\Sigma_{k,h}=V_{k-1,h}$ by Lemma~\ref{lem:infsup},
\[
\dim\bigl(\ker(\div|_{\Sigma_{k,h}})/JU_{k+2,h}\bigr)=3g.
\]
The continuous stress cohomology also has dimension $3g$ by the elasticity--de Rham cohomology correspondence~\cite{Arnold;Hu:2020Complexes,ChristiansenHu2023}. 

To prove that inclusion induces an injection, suppose $\tau_h=Jv$ with $v\in H^2(\Omega)$. On each $T$, write $\tau_h|_T=p+\sum_i a_i\boldsymbol\psi_i^k$. Since the enrichments are divergence-free, $\div p=0$. Polynomial elasticity exactness gives $p=Jq$ with $q\in\mathbb P_{k+2}(T)$. Thus $u_T:=q+\sum_i a_iv_i\in U_{k+2}(T)$ satisfies $Ju_T=\tau_h|_T$. Then $v|_T-u_T\in\mathbb P_1(T)$, so $v|_T\in U_{k+2}(T)$. 
The traces of $v$ and its first derivatives agree across every interior edge. Since the local traces are polynomial, they agree identically, including their endpoint values. Thus $v\in U_{k+2,h}$. The induced map is injective and, by equality of dimensions, an isomorphism on stress cohomology. The remaining groups agree because both Airy kernels are $\mathbb P_1$ and both divergence maps are onto. No essential boundary conditions are imposed.
\end{proof}


\begin{remark}\rm
For $k\ge2$, we also obtain the RT-type elasticity complex
\begin{equation}
\label{eq:femelasticitycomplexRT}
\mathbb P_1\hookrightarrow U_{k+3^{-},h} \xrightarrow{J}
\Sigma_{k+1^{-},h}
\xrightarrow{\div} V_{k,h}\to 0,
\end{equation}
where
\[
U_{k+3^{-},h}
=U_{k+2,h}+\Oplus_{T\in\mathcal T_h}b_T^2\mathbb P_{k-3}(T)
=U_{k+2,h}\oplus\Oplus_{T\in\mathcal T_h}b_T^2\mathbb H_{k-3}(T).
\]
Bubbles are extended by zero. Here $\mathbb H_m(T)$ consists of homogeneous polynomials of degree $m$ in coordinates centered at $\texttt v_c$, with $\mathbb H_m=\{0\}$ for $m<0$. Thus $\mathbb P_m=\mathbb P_{m-1}\oplus\mathbb H_m$. Since $b_T^2q$ and its gradient vanish on $\partial T$, its Airy image has zero normal traction and degree at most $k+1$:
\[
J(b_T^2\mathbb P_{k-3}(T))\subseteq\mathbb B_{k+1}^{\div}(T;\mathbb S).
\]
Thus $JU_{k+3^-,h}\subseteq\Sigma_{k+1^-,h}$. The RT nodal interpolant $I_{k+1^-,h}^{\div}$ uses the unchanged edge moments and interior moments against $\mathbb P_{k-1}(T;\mathbb S)$. Integration by parts gives
\[
\div I_{k+1^-,h}^{\div}\tau=Q_{k,h}\div\tau.
\]
The lifting and interpolation estimate of Lemma~\ref{lem:infsup} give $\div\Sigma_{k+1^-,h}=V_{k,h}$. Relative to \eqref{eq:femelasticitycomplexHCT}, the local dimension increments are $k-2$, $3k$, and $2(k+1)$, whose alternating sum is zero. Together with $\ker J=\mathbb P_1$, these facts prove exactness on simply connected domains.
\end{remark}

Nodal interpolation requires more regularity than the Hilbert complex \eqref{eq:2Dcomplex} provides: vertex gradients are not defined for general $v\in H^2(\Omega)$. For $\bs\sigma\in H(\div,\Omega;\mathbb S)$, the macroelement normal trace $(\bs\sigma|_T)\bs n_T$ belongs to $H^{-1/2}(\partial T;\mathbb R^2)$, which does not in general define the individual edge moments. We therefore use quasi-interpolation.

Following the averaging construction in~\cite{BrennerSung2005,HuangHuang2011}, we define $\widetilde{I}_{k+2,h}^{H^2}: H^2(\Omega)\to U_{k+2,h}$ by the DoFs \eqref{HCTfeDoF}, with \eqref{HCTfeDoF2} replaced by
$$
\nabla (\widetilde{I}_{k+2,h}^{H^2}v)(\mathtt v)
=
\frac{1}{|\mathcal{T}_{\mathtt v}|}
\sum_{T\in\mathcal{T}_{\mathtt v}}
\nabla (I_{h}^{L}v)|_T(\mathtt v),
\qquad
\mathtt v \in \Delta_0(\mathcal T_h),
$$
where $\mathcal{T}_{\mathtt v}$ is the set of triangles sharing $\mathtt v$, and $I_{h}^{L}$ is the nodal Lagrange interpolation of degree $k+2$ on the macrotriangulation $\mathcal T_h$.
Scaling and affine reproduction give the uniform bound
\begin{equation}\label{eq:scalar-quasi-bound}
\|\widetilde I_{k+2,h}^{H^2}v\|_{H^2(\Omega)}\lesssim\|v\|_{H^2(\Omega)}.
\end{equation}

Assume that $\Omega$ is simply connected. Fix a bounded linear right inverse $R:L^2(\Omega;\mathbb R^2)\to H^1(\Omega;\mathbb S)$ of $\div$, as in Lemma~\ref{lem:infsup}; for example, take $Rf$ to be the minimum-$H^1$-norm solution of $\div\tau=f$. Set
\[
Z:=\{w\in H^2(\Omega):(w,p)_\Omega=0\ \forall p\in\mathbb P_1\},
\qquad A:=(J|_Z)^{-1}:\ker\div\longrightarrow Z.
\]
Continuous exactness and the second-order Poincar\'e inequality imply $\|A\sigma\|_{H^2}\lesssim\|\sigma\|_0$. For $\tau\in H(\div,\Omega;\mathbb S)$, write
\[
\tau_1:=R\div\tau,\qquad
w:=A(\tau-R\div\tau),\qquad \tau=\tau_1+Jw,
\]
and define the nonlocal linear operator
\begin{equation}\label{eq:fixed-quasi-interpolation}
\widetilde I_{k,h}^{\div}\tau
:=I_{k,h}^{\div}R\div\tau
+J\widetilde I_{k+2,h}^{H^2}A(\tau-R\div\tau).
\end{equation}
Fixing $R$ and $Z$ makes this operator well defined.

\begin{proposition}\label{thm:commuting-quasi-interpolation}
Let $k\geq2$ and let $\Omega$ be simply connected. The operators defined above are bounded uniformly in $h$ in $H^2$ and $H(\div)$, respectively, and satisfy
\begin{align}
\widetilde I_{k,h}^{\div}(Jv)
&=J\widetilde I_{k+2,h}^{H^2}v,
&&v\in H^2(\Omega),\label{eq:quasi-commute-J}\\
\div\widetilde I_{k,h}^{\div}\tau
&=Q_{k-1,h}\div\tau,
&&\tau\in H(\div,\Omega;\mathbb S).\label{eq:quasi-commute-div}
\end{align}
Thus the following diagram commutes:
\begin{equation}\label{eq:hilbert-commuting-diagram}
\begin{array}{c}
\xymatrix{
\mathbb P_1 \ar[r]^-{\subset} \ar[d]^{\mathrm{id}}
& H^2(\Omega) \ar[d]^{\widetilde I_{k+2,h}^{H^2}} \ar[r]^-{J}
& H(\div,\Omega;\mathbb S) \ar[d]^{\widetilde I_{k,h}^{\div}} \ar[r]^-{\div}
& L^2(\Omega;\mathbb R^2) \ar[d]^{Q_{k-1,h}} \ar[r]
& 0
\\
\mathbb P_1 \ar[r]^-{\subset}
& U_{k+2,h} \ar[r]^-{J}
& \Sigma_{k,h} \ar[r]^-{\div}
& V_{k-1,h} \ar[r]
& 0
}
\end{array}
\end{equation}
\end{proposition}
\begin{proof}
The bounds for $R$, $A$, \eqref{eq:stress-interpolation-bound}, and \eqref{eq:scalar-quasi-bound} yield
$\|\widetilde I_{k,h}^{\div}\tau\|_{H(\div)}\lesssim\|\tau\|_{H(\div)}$.
For $\tau=Jv$, one has $R\div\tau=0$ and $A(Jv)=v-Q_{\mathbb P_1}v$, where $Q_{\mathbb P_1}$ is the $L^2$ projection onto $\mathbb P_1$. Affine reproduction and $J\mathbb P_1=0$ give \eqref{eq:quasi-commute-J}. Applying $\div$ to \eqref{eq:fixed-quasi-interpolation} and using \eqref{eq:commuprop1} gives \eqref{eq:quasi-commute-div}.
\end{proof}
On multiply connected domains, cohomology preservation follows from Theorem~\ref{thm:femelasticitycomplexHCT}; the above definition of $A$ requires simple connectivity.\par

\begin{remark}[Low-order finite element elasticity complexes]\label{rem:low-order-complexes}
The full Johnson--Mercier stress space is
\[
\Sigma_{\rm JM}(T)=H(\div,T;\mathbb S)\cap\mathbb P_1(T^{\rm R};\mathbb S),
\qquad \div\Sigma_{\rm JM}(T)=\mathbb P_0(T^{\rm R};\mathbb R^2).
\]
Here polynomials on $T^{\rm R}$ are defined piecewise. The HCT--Johnson--Mercier connection is discussed in~\cite{ArnoldWinthe2003finite}. The full local complex is exact~\cite[Theorem 4]{ChristiansenHu2023}:
\begin{equation}\label{eq:JM-local-complex}
\mathbb P_1\hookrightarrow U_3(T)\xrightarrow{J}\Sigma_{\rm JM}(T)
\xrightarrow{\div}\mathbb P_0(T^{\rm R};\mathbb R^2)\to0.
\end{equation}
The potential, stress, and displacement spaces have dimensions $(12,15,6)$, respectively. The displacement space used in the original mixed method is $\mathbb P_1(T;\mathbb R^2)$, which also has dimension $6$ but differs from $\mathbb P_0(T^{\rm R};\mathbb R^2)$~\cite{JohnsonMercier1978,GopalakrishnanGuzmanLee2025}. 
Christiansen--Hu~\cite[Remark 32]{ChristiansenHu2023} give the reduced elasticity complex
\[
\mathbb P_1\hookrightarrow U_2(T)\xrightarrow{J}\Sigma_{\rm CH}^{\rm red}(T)
\xrightarrow{\div}\curl(C^0(T)\cap\mathbb P_1(T^{\rm R}))\to0,
\]
with dimensions $(9,9,3)$. Here $\Sigma_{\rm CH}^{\rm red}(T)$ is the subspace of $\Sigma_{\rm JM}(T)$ with divergence in the displayed terminal space and constant tangential traction on each macroedge.
\end{remark}


\begin{remark}[A degree-nested companion family]\label{rem:nested-companion}
Let $\overline v_i$ and $\overline{\boldsymbol\psi}_i=J\overline v_i=\boldsymbol\psi_i^2$, $i=0,1,2$, be the enrichments from Lemma~\ref{lem:characterization2d} with $k=2$. For $k\ge2$, define
\[
\begin{aligned}
\widehat U_{k+2}(T)&:=\mathbb P_{k+2}(T)\oplus
\operatorname{span}\{\overline v_0,\overline v_1,\overline v_2\},\\
\widehat\Sigma_k(T)&:=\mathbb P_k(T;\mathbb S)\oplus
\operatorname{span}\{\overline{\boldsymbol\psi}_0,\overline{\boldsymbol\psi}_1,\overline{\boldsymbol\psi}_2\}.
\end{aligned}
\]
These spaces have the same dimensions and the same unisolvent DoFs \eqref{HCTfeDoF} and \eqref{DoF} as $U_{k+2}(T)$ and $\Sigma_{k,\psi}^{\div}(T;\mathbb S)$, respectively.

Indeed, the internal-edge jumps in \eqref{eq:difference}, with $k=2$, show that the fixed stress enrichments are independent modulo $\mathbb P_k(T;\mathbb S)$. If $\sum_i a_i\overline v_i\in\mathbb P_{k+2}(T)$, applying $J$ gives $\sum_i a_i\overline{\boldsymbol\psi}_i\in\mathbb P_k(T;\mathbb S)$, forcing all coefficients to vanish. Moreover, zero boundary traction eliminates the coefficients of the fixed stress enrichments, and the remaining polynomial stress is determined by the interior moments, as in Theorem~\ref{thm:unisolv_sigma_kpsi}. For the scalar space, $u|_e\in\mathbb P_{k+2}(e)$ and $\partial_nu|_e\in\mathbb P_{k+1}(e)$; hence vanishing boundary DoFs \eqref{HCTfeDoF1}--\eqref{HCTfeDoF4} imply zero function and normal-derivative traces. Consequently, the Airy image has zero boundary traction. The preceding stress argument then eliminates the fixed enrichments, leaving a function in $b_T^2\mathbb P_{k-4}(T)$, which vanishes by the interior moments.

Since $J\overline v_i=\overline{\boldsymbol\psi}_i$, $\div\overline{\boldsymbol\psi}_i=0$, and the polynomial divergence map is onto, polynomial elasticity exactness yields the local exact sequence
\[
\mathbb P_1(T)\hookrightarrow\widehat U_{k+2}(T)
\xrightarrow{J}\widehat\Sigma_k(T)\xrightarrow{\div}
\mathbb P_{k-1}(T;\mathbb R^2)\longrightarrow0.
\]
Assembly with the same shared DoFs gives conforming spaces that, on any fixed macrotriangulation, are nested in polynomial degree:
\[
\widehat U_{k+2,h}\subseteq\widehat U_{k+3,h},\qquad
\widehat\Sigma_{k,h}\subseteq\widehat\Sigma_{k+1,h}.
\]
The companion family coincides with the original family for $k=2$, but differs from it for every $k>2$. Indeed, the nonzero enrichment jumps in the original family have degree $k$ on the internal edges by \eqref{eq:difference}, whereas those of the companion family have degree two.
\end{remark}\par

\bibliographystyle{abbrv}
\bibliography{references,./paper,macroelement_spline_references}


\end{document}